\documentclass[11pt,letterpaper]{amsart}
\usepackage[utf8]{inputenc}
\usepackage{amsmath,amsfonts,amssymb,amsthm}
\usepackage{mathtools}
\usepackage{amssymb}
\usepackage{enumerate}
\usepackage{xcolor}
\usepackage{comment}
\usepackage{graphicx}
\usepackage{color}
\usepackage{ulem}
\usepackage{hyperref}

\newtheorem{theorem}{Theorem}[section]
\newtheorem{lemma}[theorem]{Lemma}
\newtheorem{corollary}[theorem]{Corollary}

\theoremstyle{definition}
\newtheorem{definition}[theorem]{Definition}

\newtheorem{proposition}[theorem]{Proposition}
\theoremstyle{remark}

\numberwithin{equation}{section}

\newcommand{\be}{\begin{equation}}
	\newcommand{\ee}{\end{equation}}
\newcommand{\wt}{\widetilde}
\newcommand{\ep}{\varepsilon}
\newcommand{\R}{\mathbb{R}}
\newcommand{\C}{\mathbb{C}}

\title{On the finiteness {issue} of four-body balanced configurations in the plane}

\author{Yuchen Wang}
\address{Institute of Mathematics, University of Augsburg, Augsburg, 86159, Germany and School of Mathematics Science, Tianjin Normal University, Tianjin, 300387, China }
\email{yuchen.wang@uni-a.de}

\author{Lei Zhao}
\address{Institute of Mathematics, University of Augsburg, Augsburg, 86159, Germany}
\email{lei.zhao@uni-a.de}

\date{\today}

\begin{document}
	\maketitle
	\begin{abstract}
	We show that the number of $\mathbf{S}$-balanced configurations of four bodies in the plane is finite, provided that the symmetric matrix $\mathbf{S}$ is close to a numerical matrix. 
	\end{abstract}
	
\section{Introduction} \label{S:Intro}

The Newtonian $N$-body problem studies motions of $N$ point masses, {positioned at} $q_{1}, \cdots, q_{N}$ in the Euclidean space $\R^{d}, 1 \leq d \leq 3$ with masses $m_{1}, \cdots, m_{N}>0$ under mutual Newtonian attractions. The dynamics is governed by the system of differential equations 
\be \label{E:Nbody-0}
\ddot{q}_j(t) = \sum_{1 \leq i \leq N, i \neq j} m_i \frac{q_i-q_j}{\| q_i-q_j \|^3}, \quad 1 \leq j \leq N,
\ee
defined on the configuration space
\[
\mathcal{Q}:= \{(q_1,\ldots,q_N) \in \mathbb{R}^{Nd} \setminus \Delta \}, \;\; \Delta:=\{(q_1,\ldots,q_N)| \; \exists i \neq j,\; s.t. \; q_i=q_j \}. 
\]
{Note that we have set the gravitational constant to $1$.}

Define the mass matrix $\mathcal{M}:={\rm diag} \big(m_1 I_{d\times d}, \ldots, m_N I_{d\times d}\big)$ and the Newtonian potential
\be
V(\mathbf{q}) = -U(\mathbf{q})
\ee
with
\be
U(\mathbf{q}):=\sum_{1 \leq i \neq j \leq N} \frac{m_i m_j}{\| q_i-q_j \|}, \; \mathbf{q}:=(q_1,\ldots,q_N),
\ee
traditionally called the force function. System \eqref{E:Nbody-0} can be reformulated as 
\be  \label{E:Nbody}
\mathcal{M} \ddot{\mathbf{q}} = \nabla_i U(\mathbf{q}).
\ee
%within which the gravitational constant has been set to a unit. \\

{The two-body problem was initially proposed and solved by Newton himself \cite{Newton1687}.} The N-body problem with $N \geq 3$ is non-integrable \cite{Bruns1887}, \cite{Poincare}, \cite{Tosel},  exhibiting rich and complex dynamics.

%Locating solutions as explicit as possible is quite a natural way to study such a complicated system. 
To {study} such a system, it is natural to seek solutions that are as explicit as possible. 
One significant class of such solutions is {given by} relative equilibrium motions, {along which} the shape of the configuration formed by the point masses remains unchanged over time. This is equivalent to rigid motion solutions, described by the invariance of the mutual distances between the particles, so that the whole configuration moves as a rigid body. 

\textbf{We fix the center of mass at the origin in all of the following discussions.}

If the system is planar, $d=2$, a configuration of $N$ point masses in $\R^{2}$ gives a solution of this type if and only if 
\be \label{E:CC-1}
\lambda\, q_j = \sum_{1 \leq  k \leq n, k \neq j} m_k \frac{q_{k}-q_{j}}{\|q_k-q_j\|^{3}}, \quad  j=1,\ldots,N.
\ee
This means that the force exerted on $q_{j}$ is proportional to the relative position of $q_{j}$ with respect to the center of mass of the system, with a common factor $\lambda$ for all $j$.  One imagines that if we release such a configuration at rest, then they will move together toward their center of mass ``centrally'' with shrinking size and eventually {end} up at a total collapse. For this last reason, such configurations are called {\it central configurations} (CC's for short). 

Planar central configurations also give rise to homographic solutions along which each particle moves along a Keplerian orbit and the whole configuration is only rotated and dilated. So each central configuration leads to a family of solutions of the planar N-body problem which are as explicit as solutions of the planar Kepler problem. The remaining step in deriving such a solution with an explicit time parametrization is the inversion of the Kepler equation, which \emph{nevertheless} cannot be achieved in closed form {in general}.  

Observe \eqref{E:CC-1} is valid in all dimensions, so the definition of a central configuration is not restricted only to the planar problem. We may as well consider central configurations in $\R^{d}$ also for non-physical dimensions $d \ge 4$. For large {numbers} of $N$ and $d$, the  existence of {a} large number of central configurations are expected. Some of these central configurations that one can identify by symmetry considerations are remarked in \cite{HeckmanZhao2016,Montaldi2015}.

From a variational viewpoint, central configurations are critical points of the {potential function} $U(\mathbf{q})$ restricted to the inertia ellipsoid 
\be
\{ I(\mathbf{q}) := \sum_{j=1}^N m_i \|q_i\|^2 =1 \}. 
\ee
Central configurations have additional significance: They are limiting configurations for N-body total collisions, as well as for complete parabolic motions (see c.f. \cite{Chenciner1998}). They are also critical points at which the joint level sets of all known conserved quantities, traditionally called integral manifolds of the system, may bifurcate. \\

In the case $N=3, d=2$, all central configurations are well-understood for any triple of positive masses $(m_{1}, m_{2}, m_{3})$. There are three collinear ones identified by Euler \cite{Euler}, and two equilateral ones discovered by Lagrange \cite{Lagrange1772}. Here the count is on the equivalent classes of configurations up to {rotations and dilations}, which are obvious symmetries of \eqref{E:CC-1}. The case $d=1$ with any $N$ {is} analyzed by Moulton \cite{Moulton1911} as a generalization of Euler's work. The next simplest case $N=4, d=2$ already brings substantial difficulties and a complete understanding is in general out of reach. Except for certain cases with specific masses \cite{Albouy1995}, \cite{Xia1991}, it seems hopeless to identify all central configurations in this case. The much {more} modest question of whether there are only finitely {many planar} central configurations with 4 positive masses has only been answered in 2006 by Hampton and Moeckel with a computer-assisted proof \cite{HamMoecke2006}. In 2012, Albouy and Kaloshin \cite{AlbouyKaloshin2012} gave a simpler proof without the need of using computers as an overture for their proof of the finiteness of 5-body central configurations in the plane for generic positive masses.

The finiteness of central configurations with an arbitrary number of positive masses was proposed by Wintner \cite{Wintner1941} and was listed by Smale \cite{Smale1998} among his mathematical problems for this century. Besides the works of Hampton-Moeckel and Albouy-Kaloshin, we cite \cite{JenLey2023} and \cite{ChangChen1}, \cite{ChangChen2} for recent different attempts on the finiteness problem for 5 and 6-particles in the plane.  \\

In \cite{Alain1998}, {Albouy and Chenciner} studied rigid body motions of the N-body problem in an Euclidean space of arbitrary dimension. The configurations admitting such motions, called \emph{balanced configurations}, are more general than {central configurations}.

In the plane, balanced configurations (BC's for short) are defined as those {solve the system of equations}
\be \label{E:BC-1}
\lambda \, \mathbf{S} q_j = \sum_{1 \leq  k \leq n, k \neq j} m_k \frac{q_{k}-q_{j}}{\|q_k-q_j\|^{3}}, \quad  j=1,\ldots,N,
\ee 
for a positive-definite symmetric matrix $\mathbf{S}_{2 \times 2}$. A configuration solving this system of equations with a given $\mathbf{S}$ is referred to as a $\mathbf{S}$-balanced configuration,  or $\mathbf{S}$BC for short \cite{Moeckel2014}. 

As only the product $\lambda \mathbf{S}$ appeas in \eqref{E:BC-1}, we impose a normalization condition on $\mathbf{S}$: We set 
$$tr(\mathbf{S})=2.$$ 
This is in consistence with \eqref{E:CC-1} in which it is natural to {interpret} $\mathbf{S}=\hbox{Id}$. We write $1+\eta, 1-\eta$, in which  $\eta \in [0,1)$, for the eigenvalues of $\mathbf{S}$.
%,  {\it balanced configuration} which is defined as a solution of the algebraic equations
%is introduced due to increasing complexity of the special orthogonal groups in the higher dimensional Euclidean spaces.
%to describe rigid motions of point masses in the higher dimensional spaces.
%See also the note of Moeckel \cite{Moeckel2014} for details. Such a configuration is usually referred to $\mathbf{S}$-balanced configurations for emphasizing the dependence on the $\mathcal{S}$. 

%Suppose $O \in SO(2)$ and $\mathcal{O} :={\rm diag} (O,\ldots,O)$. Clearly $\mathcal{O} \mathbf{q}$ is an $O^t \mathbf{S} O$-balanced configuration, that is, balanced configuration is not isolated thus we should count appropriate representatives to  enumerate the number of balanced configurations precisely.
Rotating an $\mathbf{S}$-balanced configuration in general does not lead to a $\mathbf{S}$-balanced configuration again, unless if $\mathbf{S}$ is numerical, and in this case the configuration is a central configuration. So fixing $\mathbf{S}$ in the case that $\mathbf{S}$ is not numerical  (\emph{i.e.} has distinct eigenvalues) {eliminates} rotations. 

{Variationally, we know \cite{Alain2011,Moeckel2014} that} $\mathbf{S}$-balanced configurations are the critical points of Newtonian potential restricted on the weighted 
{$\mathbf{S}$}-{moment of inertia} ellipsoid 
\[
\{ \mathbf{q} \in \mathbb{R}^{Nd} | I_\mathbf{S}(\mathbf{q}): = \sum_{j=1}^N m_j \mathbf{S} q_j \cdot q_j  =1 \}.
\]

For a non-numerical $\mathbf{S}$, non-degenerate critical points of $U(\mathbf{q})$ are isolated, in contrast to the situation for central configurations which always stay in families due to the rotation symmetry. {Morse-theoretical studies} on the $\mathbf{S}$-balanced configurations has been discussed in \cite{Moeckel2014}, \cite{Portaluri2023}. 	     
	              
The balanced configurations of three particles are {explicitly described} in \cite{Alain1998}. 
{Few information is known for balanced configurations with more particles. Due to the apparent similarity of their defining equations, it is also natural to ask} for finiteness of $\mathbf{S}$-balanced configurations for a given set of mass particles. This is the topic of this article. {In the plane, balanced configurations first appeared in the work of Moeckel \cite{Moeckel1997} when he studied clusters of small masses in central configurations. So the finiteness of balanced configuration is {well-related} to the finiteness of central configurations.}

Compared to the analysis of $CC's$, considering $\mathbf{S} BC's$ in the plane with a non-numerical $\mathbf{S}$ eliminates rotation invariance and thus makes the situation less degenerate. We shall see in the analysis that this brings new unforeseen challenges as well.

Here {is} our main result:
\begin{theorem} \label{T:main}
For any positive masses $m_{1}, m_{2}, m_{3}, m_{4}$, there exists $0<\eta_{0}<1$, such that for any $0<\eta<\eta_0$, the system \eqref{E:BC-1} has only finitely many solutions up to dilations in the plane.
\end{theorem}

Note that this does not directly follow from the finiteness of central configurations as {in general we do not have non-degeneracy for them}.

{The result is obtained by}  the approach of Albouy-Kaloshin \cite{AlbouyKaloshin2012}. After embedding the system into an algebraic system,  should an infinite number of planar $\mathbf{S} BC's$ counted up to dilations exist for a given set of masses of four bodies, then there are singular sequences formed by these configurations along which $\max \{r_{ij}, r_{ij}^{-1}\} \to \infty$. The possibilities are classified by two-color singular diagrams, which either holds for some particular combination of masses, or can be further eliminated by other reasons. Difficulty appears when there are singular diagrams which does not give rise to mass conditions, and cannot be eliminated as well. For $\mathbf{S}BC$'s with $0<\eta<1$, an additional difficulty rises: due to the lack of certain normalization condition, there are more singular diagrams which {do} not lead to mass conditions as compared to the list of diagrams, and in general, it seems difficult to eliminate them. As they do not appear in the analysis of central configurations, corresponding to the case $\eta=0$, they can be avoided for small $\eta$. This leads to the perturbative statement in Theorem \ref{T:main}. {As new diagrams appear as {compared} to \cite{AlbouyKaloshin2012},} we introduce additional rules {complementing} those of Albouy-Kaloshin {in our case.} 
 
{This perturbative finiteness result can be alternatively proven in a simpler manner which can be directly extended to the case $N$-particles in the plane, which we present toward the end of this article.} % result concerning finiteness of $\mathcal{S}$-balanced configurations of the $N$-body problem if there is finitely many central configurations.
Already for the four-body case, it would be nice to remove the condition that $\eta$ should be close to zero. {We conjecture }that {finiteness holds true for {generic choice of positive masses and  $0<\eta<1$}}. {Our complete analysis of singular diagrams serves as a starting point toward this more general goal. }%but probably different methods will be required for this.
 
This paper is organized as follows: In Section \ref{S:Pre}, we sketch the approach of Albouy-Kaloshin in our case. In Section \ref{S:AK-BC} we deduce rules for singular diagrams, which are used in  Section \ref{S:N=4} to obtain a complete list of possible singular diagrams. These singular diagrams are studied in Section \ref{S:mass} in which we show that most of them are realizable only if certain mass conditions are satisfied. 
In Section \ref{S:Finitness}, we study the finiteness issue of $\mathbf{S}BC$'s in the planar four-body problem and discuss technical difficulties, and we prove Theorem \ref{T:main} in Section \ref{S:Proof}. In  Section \ref{S:N-body} we discuss the case of N-bodies in the plane.

\section{The method \`{a} la Albouy-Kaloshin} \label{S:Pre}

The Equations \eqref{E:CC-1} and \eqref{E:BC-1} can both be embedded into a system of algebraic equations, which can be consequently complexified by considering the coordinates $(x_j,y_j)$ of the $j$-th body as complex. In this way one factorizes the squared mutual distances $r_{ij}^{2}$. Let 
\[
z_j := x_j + \sqrt{-1}\, y_j, w_j = x_j - \sqrt{-1}\, y_j, \quad j=1,\ldots,N.
\]
Then

\[
r_{ij}^{2} = z_{ij} w_{ij}%, \quad r_{ij}=r_{ji}
\]
and
\[
r_{ij}{=r_{ji}} = z_{ij}^{\frac{1}{2}} w_{ij}^{\frac{1}{2}}.
\]
The square root of a complex quantity is not uniquely defined and {a} choice is needed. Nevertheless, this ambiguity does not affect what follows: Nothing changes when we proceed with \emph{any} choice of the branch. 

We further set
\be \label{E: ZW from zw}
Z_{ij} =z_{ij}^{-\frac{1}{2}} w_{ij}^{-\frac{3}{2}}, \quad W_{ij} = w_{ij}^{-\frac{1}{2}} z_{ij}^{-\frac{3}{2}}.
\ee
Since $Z_{ij} =r_{ij}^{-3} z_{ij}, W_{ij}= r_{ij}^{-3} w_{ij}$, {we have} 
$Z_{ij}=-Z_{ji}, W_{ij}=-W_{ji}$.

Then \eqref{E:BC-1} can be written as 
\be \label{E:BC-ZW}
\begin{split}
z_j & = \frac{1}{1-\eta^2} \sum_{k \neq j} m_k Z_{kj} - \frac{\eta}{1-\eta^2} \sum_{k \neq j}  m_k W_{kj}, \\
w_j & = \frac{1}{1-\eta^2} \sum_{k \neq j} m_k W_{kj} - \frac{\eta}{1-\eta^2} \sum_{k \neq j}  m_k Z_{kj}. 
\end{split} 
\ee
In particular, when $\eta =0$ we recover {the equations for central configurations}
\be \label{E:CC}
\begin{split}
	z_j & =  \sum_{k \neq j} m_k Z_{kj}{,}  \\
	w_j & =  \sum_{k \neq j} m_k W_{kj}.
\end{split} 
\ee

{We see that the} case $\eta=0$ (central configurations) and the case $0<\eta<1$ are rather different. In the former case the system formally decouples into a set of equations for $\mathcal{Z}$ and another set of equations for $\mathcal{W}$. In the latter case the system is coupled.

{Let ${\mathcal{N} = N(N+1)}$.} This system of equations defines {a complex} algebraic variety  $\mathcal{A}  {\subset \mathbb{C}^{\mathcal{N}}}$. The elements are denoted by
\[
\begin{split}
\mathcal{Q}:& =\left(\mathcal{Z},\mathcal{W} \right)\\
&= (z_1,\ldots,z_N,Z_{12},\ldots,Z_{(N-1)N},w_1,\ldots,w_N,W_{12},\ldots,W_{(N-1)N}),
\end{split}
\]
in which 
\[
\begin{split}
\mathcal{Z}&=(z_1,\ldots,z_N,Z_{12},\ldots,Z_{(N-1)N}), \\ \mathcal{W}&=(w_1,\ldots,w_N,W_{12},\ldots,W_{(N-1)N}),
\end{split}
\]
 are called the $\mathcal{Z}$- and $\mathcal{W}$-component of $\mathcal{Q}$ respectively. We write that $\|\mathcal{Z}\|$ and  $\|\mathcal{W}\|$ for the absolute values of the maxima of their components respectively: { $$\|\mathcal{Z}\| :=\max_{1 \leq j \leq  \frac{\mathcal{N}}{2}} |Z_j|, \quad \|\mathcal{W}\| :=\max_{1 \leq j \leq \frac{\mathcal{N}}{2}} |W_j|.$$}
  
In a topological space, a set is called locally compact if it is the intersection of an open and a closed set. A set {is} called constructible if it is a finite union of locally compact sets. The space $\C^{n}$ is equipped with the Zariski topology, for which the {closed} sets are zeros of polynomial systems. The following basic version of Chevalley's theorem (See \cite{AlbouyKaloshin2012}. A standard reference for a much more general statement of this theorem is \cite{EGAIV}.) is a basic tool {to establish} finiteness of central/balanced configurations. 
%at the heart of arguments in the effects to pursue the finiteness of central configurations in the plane.

\begin{theorem}[Chevalley]%\marginpar{Add CITATION: \rd{EGA is too much....}}
The image of a constructible subset of $\mathbb{C}^n$ by a polynomial map $\mathcal{P}:\mathbb{C}^n \rightarrow \mathbb{C}^m$ is again constructible. 
\end{theorem}   

Note that a {constructible} set in $\C$ is either itself a finite set or has a finite complement. If the image of a polynomial function $\mathcal{P}:\mathbb{C}^n \rightarrow \mathbb{C}$ has a finite complement, then we call it a dominating function.
 
To show that $\mathcal{A}$ is a finite set, the idea of Albouy and Kaloshin is the following: Assume $\mathcal{A}$ consists of infinitely many points. Then at least two mutual distances $r_{ij}, r_{kl}$ take infinitely many values, so they are dominating. Consequently we can find in $\mathcal{A}$ a sequence $(\mathcal{Z}^{n}, \mathcal{W}^{n})$ along which $\max \{\|\mathcal{Z}^{n}\|, \|\mathcal{W}^{n}\|\} \to \infty$. There are strong constraints on these singular sequences, characterized by their \emph{singular diagrams}. Some of these singular diagrams can be realized only if the masses satisfy an algebraic condition.

So if there are no Albouy-Kaloshin {singular} diagrams, then we have finiteness of $\mathcal{A}$ for all masses. {If all diagrams lead to a mass condition, then we have finiteness of $\mathcal{A}$ at least for generic (indeed, Zariski-open) masses. We obtain this result even with the appearance of other diagrams,} when these can be excluded for other reasons. Further elimination of diagrams {improves} the obtained results by dropping mass conditions. In the good case, the finiteness of $\mathcal{A}$ is obtained once all these diagrams are eliminated. 

Our analysis follows this identification-elimination procedure.

\section{Singular sequences of balanced configurations} \label{S:AK-BC}   

If the algebraic variety $\mathcal{A}$ has infinitely many points, then it has one or several ends tending to infinity, which are called singularities of $\mathcal{A}$. The complexity of $\mathcal{A}$ makes it hard to study these singularities directly. Singular sequences capture much less {information}, nevertheless they are much simpler to analyze and remain sufficiently effective. They are also realized graphically by two-color diagrams, one for $\mathcal{Z}$-component and the other one for $\mathcal{W}$-component with vertices representing the particles. {In the $z$-diagram, a $z$-circle is added to the particle $i$ if $z_{i}$ is of leading order, \emph{i.e.} if $|z_{i}^{n}|$ and $\|\mathcal{Z}^{n}\|$ are of the same order as $n \to \infty$. }
%A circle is drawn around the vertex $i$ if $z_{i}$ is of leading order. 
A $z$-stroke between particles $i$ and $j$ is drawn if $Z_{ij}$ is of leading order. The same is applied to the $w$-diagram. As we consider the double-colored diagram, we say there is a $z$-edge between the vertices $i, j$ if they are connected only via a $z$-stroke. We say there is a $w$-edge between the vertices $i, j$ if they are connected only via a $w$-stroke. We say there is a $zw$-edge between  the vertices $i, j$ if there are both $z$ and $w$-strokes connecting them. 

Not all circled two-color single-edge diagrams correspond to our singular sequences. {For the analysis of $CC$'s,} a list of rules that these diagrams need {to satisfy} were introduced in \cite{AlbouyKaloshin2012}. Many of these rules are carried directly to {our analysis of $BC$'s}. Also, there are more rules needed to be introduced. 

Here is a remark showing a technical difference between the analysis of CC where $\eta=0$ and the {more general} case $\eta>0$: The system of equations \eqref{E:BC-ZW} admits a dilation symmetry {of the type} $(\mathcal{Z}, \mathcal{W}) \mapsto (\lambda \mathcal{Z}, \mu \mathcal{W})$ if and only if $\eta=0$. In \cite{AlbouyKaloshin2012}, this is used to normalize singular sequences for central configurations {to always have} $\|\mathcal{Z}^{n}\|=\|\mathcal{W}^{n}\|$. {\textbf{We drop the superscript $n$ in the following analysis.}} We say that $\mathcal{Z}, \mathcal{W}$ are {of equal order} in this case{, and write $\|\mathcal{Z}\|=\|\mathcal{W}\|$}. When $\eta >0$, the lack of this dilation symmetry forces us to also consider singular sequences and the corresponding singular diagrams for which $\mathcal{Z}$ and $\mathcal{W}$ are not of {equal order}. We therefore divide our analysis accordingly {for the two different cases}:
 
\begin{itemize}
	\item $\mathcal{Z},\mathcal{W}$ are of {equal} order. We set both of them 
%$\|\mathcal{W}\|$ and $\|\mathcal{Z}\|$ 
to be of order $\ep^{-2}$,  
%\simeq |\mathcal{Z}| \sim \ep^{-2}$, 
 in which {$\ep=\ep(n) \to 0$} {serves as an order-indicating parameter in the limiting process; }
	
	\item $\mathcal{Z}$ and $\mathcal{W}$ {are not of equal} order. We assume $\|\mathcal{Z}\|$ is of order $\ep^{-2}$ and $ \|\mathcal{W}\|/\|\mathcal{Z}\| \rightarrow 0$.

	%Without loss of generality we assume  $\|\mathcal{W}\| \prec \|\mathcal{Z}\| = \ep^{-2}$.
\end{itemize}

\subsection{ {{Equal-order} singular diagrams}} 

The following notations on the orders of quantities along the singular sequences {will} be frequently used: 
\be
\begin{split}
a \sim b & \Leftrightarrow \text{$\frac{a}{b} \rightarrow 1$}, \;\;
a \prec b \Leftrightarrow \text{ $\frac{a}{b} \rightarrow  0$ }, \;\;
a \succ b  \Leftrightarrow \text{ $\frac{b}{a} \rightarrow  0$ }, \\
a \preceq b &  \Leftrightarrow \text{there exists $C>0$ such that $\|a\| \leq C \|b\|$},\\
a \succeq b &  \Leftrightarrow \text{there exists $C>0$ such that $\|b\| \leq C \|a\|$}, \\
a \simeq b & \Leftrightarrow a \preceq b \;\; \&  \; \; a \succeq b.
\end{split}
\ee

We have set {$ \|\mathcal{W}\| \simeq \|\mathcal{Z}\| \sim \ep^{-2}$.} We assume $0 < \eta <1$, so the $\mathbf{S}BC$ under consideration is assumed to be not central.

{Remind that the center of masses is set at the origin.} %without loss of generality. 
\medskip

We have the following set of coloring rules on the singular diagrams:
\\

{\bf Rule 1a: } Any end of a $z$-stroke is connected further to another $z$-stroke, $w$-stroke, $z$-circle or $w$-circle. From a $z$-circled vertex there must emanate either a $z$-stroke or a $w$-stroke. In particular, if only $w$-strokes emanate from {a} vertex, then this vertex is also $w$-circled. In any diagram there exists at least one $z$-stroke and at least one $w$-stroke.

\begin{proof}
Assume there is a $z$-stroke between vertices $1, 2$ and nothing is {on vertex} $1$. This means $Z_{12}$ is of leading order, but  $z_1$ and $w_1$ are not. According to \eqref{E:BC-ZW}, these imply that $W_{12}$ is of leading order, so there is a $w$-stroke between the vertices $1, 2$. 

%Let $\sigma_1 := \frac{1}{1-\eta^2},\; \sigma_2:=-\frac{\eta}{1-\eta^2}$ and $\mathbf{B} := \begin{pmatrix}
%		\sigma_1 & \sigma_2 \\
%		\sigma_2 & \sigma_1
%	\end{pmatrix} \in GL_2(\mathbb{R}).
%	$ 
%Moreover,  we have 
%	\[
%	 \mathbf{B} \begin{pmatrix}
%		m_2 Z_{21} \\
%		m_2 W_{21} 
%	\end{pmatrix}  \sim \begin{pmatrix}
%	z_1 \\
%	w_1 
%	\end{pmatrix} \prec \frac{1}{\ep^2}. 
%	\]
% Since $\mathbf{B}$ is invertible, the inequality immediately implies $Z_{12} \prec \frac{1}{\ep^2}$. A contradicition.  
	
Suppose now the vertex $j$ is $z$-circled, but all emanating strokes are $w$-strokes. This implies that $z_{j}$ and the sum $\sum_{k \neq j}  m_k W_{kj}$ is of the same order, while the sum $ \sum_{k \neq j} m_k Z_{kj} $ in \eqref{E:BC-ZW} is of lower order. Then by \eqref{E:BC-ZW}, $w_{j}$ is of the same order as  $\sum_{k \neq j} m_k Z_{kj}$ and $z_{j}$. This means that vertex $j$ is also $w$-circled.
%	\[
%	z_j \sim \sigma_2 \left( \sum_{k \neq j} m_k W_{kj} \right) \sim \sigma_1 \left( \sum_{k \neq j} m_k W_{kj} \right) \sim w_j,
%	\]  
%hence the vertex $j$ is also $w$-circled. 
	
Finally, we assume no $w$-stroke exists in a singular diagram. By our assumption $\|\mathcal{W}\|$ is of leader order as $\|\mathcal{Z}\|$, so there must exist a $w$-circle, say around the vertex $j$. By \eqref{E:BC-ZW}, the sum $\sum_{k \neq j} m_k Z_{kj}$ and thus $z_{j}$ is also of leading order. So $j$ is $z$-circled as well. As there are no $w$-strokes in the diagram, from vertex $j$ there must emanate a $z$-stroke to, say vertex $k$. So $Z_{jk}$ is of leading order. We observe 
	\be
	\begin{split}
	z_{ij} & \simeq  \left( \sum_{k \neq i}m_k Z_{ki} - \sum_{k \neq j} m_k Z_{kj} \right), \\
		w_{ij}  & \simeq  \left( \sum_{k \neq i}m_k Z_{ki} - \sum_{k \neq j} m_k Z_{kj} \right),
	\end{split}
	\ee
so $z_{ij}$ and $w_{ij}$ are of the same order, and consequently from \eqref{E: ZW from zw}, $W_{jk}$ is of the same order as $Z_{jk}$, so there is a $w$-stroke between vertices $j$ and $k$. Contradiction. So there must exist a $w$-stroke, and by the same reasoning, also a $z$-stroke.
%\[
%z_j = \sigma_1 \left( \sum_{ k \neq j} m_k Z_{kj} \right) \sim \frac{\sigma_1}{\sigma_2} w_j,
%\]
%that is, body $j$ is also $z$-circled provided it is $w$-circled.

%Moreover, there exists maximal $z$-strokes. Suppose $Z_{ij}$ is a maximal $z$-stroke, we have
%	\be \label{E:z-only}
%	\begin{split}
%		z_{ij} & \sim  \sigma_1 \left( \sum_{k \neq i}m_k Z_{ki} - \sum_{k \neq j} m_k Z_{kj} \right), \\
%		w_{ij}  & \sim  \sigma_2 \left( \sum_{k \neq i}m_k Z_{ki} - \sum_{k \neq j} m_k Z_{kj} \right),
%	\end{split}
%	\ee
%	then $Z_{ij} \sim W_{ij}$ due to the definitons 
%	\[
%	Z_{ij} =z_{ij}^{-\frac{1}{2}} w_{ij}^{-\frac{3}{2}}, \quad  W_{ij} =w_{ij}^{-\frac{1}{2}} z_{ij}^{-\frac{3}{2}}.
%	\]
%This  imples that if $Z_{ij}$ is a maximal $z$-stroke, then it must be a $w$-stroke. A contradiction.

\end{proof}

\begin{definition}
{Two} vertices $i$ and $j$ are called \emph{$z$-close} if $z_{ij} \prec \ep^{-2}$, \emph{i.e.} {if} $z_{ij}=z_{i}-z_{j}$ is not of leading order. By linearity, this definition extends to linear combinations of vertices.
\end{definition}
\vspace{0.3 cm}

A direct consequence of this definition is

{\bf Rule 1b: } If the bodies $k$ and $l$ are $z$-close, then they are either both $z$-circled or both not  $z$-circled.

\begin{definition}
	 A subdiagram of a singular diagram is called an \emph{isolated component} if it is a disconnected subgraph of the bi-colored diagram.
\end{definition}

%\begin{definition} 
%\end{definition}

{\bf Rule 1c: }  The center of mass of an isolated component is both $z$- and $w$-close to the origin. 

\begin{proof}
Suppose $I \subset \{1,2,\ldots,N\}$ is an isolated component. Then
\[
\sum_{j \in I} m_j z_j = \frac{1}{1-\eta^2}  \sum_{j \in I} \sum_{ k \in I,k \neq j} m_j m_k Z_{kj} - \frac{\eta}{1-\eta^2} \sum_{j \in I} \sum_{ k \in I,k \neq j} m_j m_k W_{kj}.
\]
The $Z_{jk}$'s for $j \in I, k \notin I$ are by assumption not of leading order. When $j, k \in I$, $m_{j} m_{k} Z_{kj}$ and $m_{j} m_{k} Z_{jk}$ cancel each other. {Thus} the first summand on the right-hand side of this equation {is} not of leading order. The second summand is also not of leading order by the same reasoning. Thus 
$$\sum_{j \in I} m_j z_j  \prec \ep^{-2}.$$ 

By the same reasoning, we also have 
$$\sum_{j \in I} m_j w_j \prec \ep^{-2}.$$

\end{proof}

Rule {(1b)} and Rule {(1c)} imply the following: \\

{\bf Rule 1d: } In a singular diagram or a connected component of it, the existence of a $z$-circled vertex implies the existence of another $z$-circled vertex. The $z$-circled {vertices} can not all be $z$-close together.
\medskip

%If there is a $z$-circled vertex in the diargam,  another $z$-circled vertex exists. For an isolated component, if there exists an $z$-circled vertex then there is another $z$-circled vertex. The $z$-circled bodies can not all be $z$-close together. \\
Indeed if there is only one $z$-circled vertex, or {if} all $z$-circled vertices are $z$-close to each other, then the center of mass is $z$-close to {the $z$-circled vertices}, and thus cannot be $z$-close to zero. 
 
\begin{definition} A $z$-stroke is called \emph{maximal} if its vertices are not $z$-close.
\end{definition}

{\bf Rule 1e: } An isolated component has no $z$-circled vertex if and only if it also has no maximal $z$-stroke. \\

{Indeed, no $z$-circled vertex means that all vertices are $z$-close to the origin, which is equivalent to say that there is no maximal $z$-stroke by the center of mass condition. 

%Indeed, if there is no maximal $z$-stroke in an isolated component, then all bodies in the component are $z$-close to each other. So presence of a $z$-circled vertex implies the center of masses is not $z$-close to the orgin, which contradicts Rule 1c. 

%On the other hand, if all vertices are not  $z$-circled, then any edge satisfies $z_{ij} \prec \ep^{-2}$ hence there is no $z$-maximal stroke.  \\

\medskip

{\bf Rule 1f: } {An isolated component of a bi-colored diagram cannot be single colored. }

\smallskip
The proof essentially follows from the proof of the last statement of Rule (1a). 
\\

{\bf Rule 1g: } If exactly one single-colored stroke emanates from a vertex, then this vertex must be {double-circled.} 
\begin{proof}
Without loss of generality, we suppose there is exactly one $w$-stroke {emanating from the vertex $i$ to $j$}. By \eqref{E:BC-ZW} we have $z_{i} \simeq w_{i} \simeq W_{ji}$, which are of leading order.
%If $i$ is not circled, then we have by \eqref{E:BC-ZW} that $z_{i} \sim w_{i} \sim W_{ji}$ on one hand, and $z_{i} \sim w_{i} \sim 0$ on the other hand

%\[
%\begin{split}
%0 \sim z_i & = \sigma_1 \left( \sum_k  m_k Z_{ki} \right) + \sigma_2 m_j W_{ji}  \\
%0 \sim w_i & = \sigma_2 \left( \sum_k  m_k Z_{ki} \right) + \sigma_1 m_j W_{ji}
%\end{split},
%\]
%which implies $W_{ji} \sim 0$. A contradiction.
\end{proof}

\vspace{0.5 cm}

%Two color diagrams have additional rules. In particular, one could give more precise information on the orders of $Z$ and $W$ diagrams. 

The following order estimates appeared in \cite{AlbouyKaloshin2012}. {They hold in} our case as well: 

{\bf Estimate 1: }\label{Est: 1.1} For any pair $(i,j)$, one has 
$$\ep \preceq r_{ij} \preceq \ep^{-2}, \hbox{ and } \ep^2 \preceq z_{ij} \preceq \ep^{-2}.$$ 
A $zw$-stroke between $i$ and $j$ exists if and only if $r_{ij} \sim \ep$. 
\\

{\bf Estimate 2: }\label{Est: 1.2} If there is a $z$-stroke between the vertices $k$ and $l$, then 
\[
\ep \preceq r_{kl} \preceq 1,\quad \ep \preceq z_{kl} \preceq \ep^{-2}, \quad \ep^2 \preceq w_{kl} \preceq \ep. 
\]

{As in \cite{AlbouyKaloshin2012}, the} proofs of these Estimates follow from the relations
\[
Z_{ij} = z_{ij}^{-\frac{1}{2}} w_{ij}^{-\frac{3}{2}} =r_{ij}^{-3} z_{ij} \preceq \ep^{-2}, \quad W_{ij} =w_{ij}^{-\frac{1}{2}} z_{ij}^{-\frac{3}{2}} =r_{ij}^{-3} w_{ij} \preceq \ep^{-2}.
\]

It follows also from these relations that 
\begin{itemize}
\item a $w$-edge between $i$ and $j$ is \emph{maximal} if and only if ${z_{ij} \simeq \ep^{2}}$. Similarly, a $z$-edge is maximal if and only if $w_{ij} \simeq \ep^{2}$.
\item It then follows from Estimate 1 that a $zw$-edge is never maximal. 
\item If there is a stroke between the vertices $i, j$ and $w_{ij} \simeq \ep$, then the stroke is part of a $zw$-edge.
\item if there is only a $z$-stroke between the vertices $i, j$, then $w_{ij} \prec \ep$.  if there is only a $w$-stroke between the vertices $i, j$, then $z_{ij} \prec \ep$.
\item If a single-colored stroke between the vertices $i, j$ is maximal, then $r_{ij} \simeq 1$.
\item If a single-colored stroke between the vertices $i, j$ is not maximal, then $r_{ij} \prec 1$.
\end{itemize}

\medskip

{\bf Rule 2a: } Any $zw$-edge is connected to at least another $z$-stroke or a $w$-stroke.\\

\begin{proof}If not, then both ends of a $zw$-stroke are both $z$-circled and $w$-circled, by Rule (1a) and (1d). These two vertices therefore form an isolated component. {By} Rule (1e), the edge must be a maximal $zw$-stroke, which does not exist. \\
\end{proof}

{\bf Rule 2b: } If there are two consecutive $zw$-edges, there is a third $zw$-edge closing the triangle. \\

\begin{proof} Suppose the edges between vertices $1,3$ and between $3, 2$ are $zw$-edges. By Estimate 1,  
\be 
z_{12} = z_{13} + z_{32} \preceq \ep, \quad w_{12} = w_{13} + w_{32} \preceq \ep.
\ee
By Estimate 1, we also have $r_{12}^2 =z_{12}w_{12} \succeq \ep$.  Thus $z_{12},w_{12}\sim \ep$ and there must exist a $zw$-edge between the vertices $1$ and $2$.
\end{proof}

{\bf Rule 2c: } In a triangle of edges, the edges all have the same type: All $z$-edges, all $w$-edges, or all $zw$-edges.

\begin{proof} If there are only $z$-{edges} between the vertices $1, 2$ and $2, 3$, then by Estimate 2, 
\[
w_{13} = w_{12}+w_{23} \prec \ep. 
\]  
Estimate 2 now implies that no $w$-stroke between vertices $1, 3$ exists, so the {edge} between $1, 3$ is a $z$-{edge}. 

The case with two $zw$-edges follows directly from Rule (2b).
\end{proof}

{The following Rule (2d) follows directly from Rule (1g).} \\

{\bf Rule 2d: } If from a vertex there emanate exactly one $z$-{stroke} and one $w$-{stroke}, then this vertex must be either $z$-circled, or $w$-circled, or $zw$-circled. }
%For the vertex connecting with others via exactly a $z$-stroke and a $w$-stroke. This vertex must be circled. \\
\begin{proof} 
Assume the vertex $j$ is neither $z$- nor $w$-circled. { Suppose there is only a $z$-{stroke} between $j, k$, and there is a $w$-{stroke} between $j, l$.} % and $W_{jl}$ are the edges emanating from the body $j$. 
Then from \eqref{E:BC-ZW}, we have
\[
\begin{pmatrix} 
	1 & -\eta \\
	-\eta & 1 
	\end{pmatrix}
	 \begin{pmatrix} Z_{jk} \\
	W_{jl}
	\end{pmatrix} \prec \ep^{-2},
\]   
which implies $Z_{jk},W_{jl} \prec \ep^{-2}$ since $\eta \in (0, 1)$. A contradiction. 

{Note that the subscripts $k$ and $l$ are allowed to coincide, in which we have a $zw$-edge emanating from the vertex. }

\end{proof}
%Rule 2d follows from Rule 1g directly. \\

{\bf Rule 2e: } If from a vertex there only {emanates} $z$-{edges} (or only $w$-{edges}), then this vertex is either double circled or not circled at all. \\

\begin{proof} This follows from
\[
z_j \simeq \left(\sum_{k \neq j} m_k Z_{kj} \right) \simeq w_j
\]
in the case that only $z$-{edges} emanate from vertex $j$.
\end{proof}

{\bf Rule 2f: } If from both ends of a $z$-{edge} there emanate only other $z$-{edges}, then this $z$-{edge} is not maximal.
\begin{proof}
If on the contrary the $z$-{edge}, say between the vertices $1$ and $2$, is maximal, then these vertices are not $z$-close.

From \eqref{E:BC-ZW} we get
{\small
\[
\begin{split}
z_1 & \sim  \frac{1}{1-\eta^2} \left( m_2 Z_{21} + \sum_{k \neq 1,2} m_kZ_{k1} \right), \; w_1  \sim   -\frac{\eta}{1-\eta^2} \left(m_2 Z_{21} + \sum_{k \neq 1,2} m_kZ_{k1}\right), \\
z_2 & \sim  \frac{1}{1-\eta^2} \left( m_1 Z_{12} + \sum_{k \neq 1,2} m_kZ_{k2} \right), \;
w_2  \sim    -\frac{\eta}{1-\eta^2} \left( m_1 Z_{12} + \sum_{k \neq 1,2} m_kZ_{k2} \right).
\end{split}
\]	}
Thus $z_{12} \sim w_{12}$ is of leading order. But Estimate 2 asserts that the vertices $1, 2$ should be $w$-close. Contradiction. %marginpar{\rd{The original argument seem to be lost. Please check if this is now still okay.}}
%\[
% \sigma_1 z_{12} = \sigma_2 w_{12}.
%\]
%Then if $z_{12}$ is a maximal $z$-stroke then the edge $12$ is a $zw$-edge. A contradiction.

\end{proof}

{\bf Rule 2g: } 
If there exists a $zw$-edge in a singular diagram,  there must be another $zw$-edge.  \\

\begin{proof} Estimate 1 asserts $r_{ij} \preceq \ep$ and the equality holds only if between the vertices $i, j$ there is a $zw$-edge. By the variational {characterization} of $\mathbf{S}BC$, the force function is constant along a singular sequence. {But} if there is only one $zw$-edge, say between vertices $1, 2$, we would have
\be 
U(\mathbf{q}) = \sum_{1 \leq i <j \leq N} m_im_j r^{-1}_{ij} = m_{1} m_{2} r^{-1}_{12} + \sum_{1 \leq i <j \leq N,(i,j) \neq (1, 2)} m_im_j r^{-1}_{ij} \rightarrow \infty
\ee
along the singular sequences. This is a contradiction.
\end{proof}

In general, with the same argument, we have

{\bf Observation:} If there exist $i_0,j_0$ such that $ 0 \leftarrowtail r_{i_0 j_0} \preceq r_{ij}, 1 \leq i  \neq j \leq N$, then there exists another $i_1,j_1$ satisfying $r_{i_0 j_0} = r_{i_1 j_1}$. 

\medskip

{\bf Rule 2h: } If four edges form a quadrilateral, then the opposite {edges} are of the same type. 

\begin{proof}
Suppose the vertices $1,2,3,4$ form a quadrilateral with a $z$-{edge} between $1, 2$ and a $w$-{edge} between $2, 3$. There is {an edge}  between $3, 4$ and another {edge} between $1, 4$. We assume no {$zw$-edges} {exist} in this quadrilateral.

If between $3, 4$ there is a $w$-{edge}, then the {edge} between $1, 4$ {can} not be a $w$-{edge}, since otherwise by Estimate (1.2) we would have $z_{12}, z_{23}, z_{31}  \prec \ep$, but $z_{12}=z_{14}+z_{43}+z_{32}$ and $\ep \prec z_{12}$ for the $z$-{edge} between the vertices $1, 2$. A contradiction. Thus between $1, 4$ the {edge} must be a $z$-{edge}. 

%Suppose the edge $23$ is a $w$-stroke. We claim that $34$ is a $z$-edge and $14$ will be a $w$-stroke. 

%If not, suppose $34$ is a $w$-stroke, due to Estimate (1.2) we have $z_{23},z_{34} \prec \ep$. 

%Clearly $14$ can not be a $w$-stroke. Otherwise, we would obtain
%\be \label{E:quad-1}
%\ep \prec z_{12} = z_{14}+z_{43}+z_{32} \prec \ep.  
%\ee
%A contradiction. Hence $14$ could be $z$-stroke only. 

However, in this case we have 
\be \label{E:quad-2}
z_{24} = z_{23} + z_{34} \prec \ep, \quad w_{24} = w_{21} + w_{14} \prec \ep
\ee
by Estimate 2. This implies 
\[
Z_{24} = z_{24}^{-\frac{1}{2}} w_{24}^{-\frac{3}{2}} \succ \ep^{-2},
\]
which contradicts $\|\mathcal{Z}\| \sim \ep^{-2}$. 

Therefore the {edge} between the vertices $3, 4$ must be a $z$-{edge}. Switching the role of $z$- and $w$-strokes in this argument by switching the vertices $2$ and $3$, we conclude that the {edge} between $1, 4$ is a $w$-{edge}. %\bl{Thus $34$ is a $z$-edge and $14$ is a $w$-edge.}

%Moreover, note that
%\[
%\ep \prec w_{23} = w_{21}+w_{14}+w_{43}, 
%\] 
%$w_{14} \succ \ep$. That is edge $14$ is a $w$-stroke. 

We now suppose that between $1, 2$ there is a $z$-{edge} while between $2, 3$ there is a $zw$-edge. By Estimate 2, we have $z_{23},w_{23} \sim \ep$. By Rule (2b) there cannot be a $zw$-edge between the vertices $3, 4$. If the {edge} between $3, 4$ is a $w$-{edge}, then the edge between $1, 4$ cannot be a $w$-{edge}, nor a $zw$-edge, since otherwise we have by Estimate 2 that $z_{14}, z_{43}, z_{32} \preceq \ep $ and $\ep \prec z_{12}$, which is a contradiction. So in this case the edge between $1, 4$ has to be a $z$-{edge}, but this implies again \eqref{E:quad-2} which leads to a contradiction with $\|\mathcal{Z}\| \sim \ep^{-2}$. So the edge between $3, 4$ has to be a $z$-{edge}. Since $w_{12}, w_{34} \prec \ep$, we have $w_{14} = w_{12}+w_{23}+w_{34} \sim  \ep$, and thus {between ${1, 4}$ there}  is a $zw$-edge.

By Rule (2b), there cannot exist three or more $zw$-edges in a quadrilateral. The proof is now complete.

%On the other hand, if $23$ is a $zw$-edge, due to Estimate 1.2 we have $z_{23},w_{23} \sim \ep$. If the conclusion does not hold, we assume $34$ is only a $w$-edge since the other case is excluded by Rule 2b. Again by the inequality \eqref{E:quad-1}, edge $14$ can not be a $w$ or $zw$-edge. Now suppose $14$ is a $z$-stroke, we still have \eqref{E:quad-2}, which is a contradiction with the $|\mathcal{Z}| = \ep^{-2}$. Therefore, $34$ is a $z$-stroke.  The rest is to prove $14$ is a $zw$-edge. Indeed, since $w_{12},w_{34} \prec \ep$, we have
%\[
%w_{14}= w_{12}+w_{23}+w_{34} \sim \ep, 
%\]
%which implies $14$ is a $zw$-edge since $14$ is a stroke.

\end{proof}

{\bf Rule 2i: } In an isolated quadrilateral with pairs of opposite edges of different types, if all the vertices are circled, then they are circled with the same type {of} circles: {they are either }all $z$-circles, or all $w$-circles, or all $zw$-circles.
\begin{proof}
We assume that there are $z$-strokes between the vertices $1, 2$, and between the vertices $3, 4$ and there are $w$-strokes between $1, 3$ and between $2, 4$. {Note that this includes all possible types of edges.} %\marginpar{ We could not obtain these bodies must be circled if $zw$-edges are presented, that is why I suppose they must be circled a prior.}  

 {By  assumption} %Rule (2d),  
 all the vertices are circled. Suppose $1$ is only $z$-{circled}, due to Estimate 2, vertices $1$ and $2$ are $w$-close, hence the vertex $2$ can not be $w$-circled %nor $zw$-circled, 
and hence vertex $2$ {has to be} $z$-circled. Since there is a $w$-stroke between $2, 4$, by Estimate 2, the vertices $2$ and $3$ are $z$-close. By Rule {(1b)}, the vertex $3$ is $z$-circled. %either $z$-circled or $zw$-circled. 

If the vertex $3$ is {actually} $zw$-circled, as between $3, 4$ there is a $z$-stroke, the vertices $3, 4$ are $w$-close, and thus the vertex $4$ must be $w$-circled. Being $z$-close to {the vertex} $1$, the vertex $4$ is also $z$-circled. Thus the vertex $4$ is $zw$-circled.

Then we obtained in this diagram two $w$-circled vertices $3,4$ which are also $w$-close by Estimate 2. This is disallowed by Rule {(1c)}.  

Therefore, the vertex $3$ is {only} $z$-circled. Thus the vertex $4$ is also {only} $z$-circled by Estimate 2 since it is connected to $3$ only via a $z$-stroke.   

The case that $1$ is only $w$-circled is completely similar, and in that case all vertices are $w$-circled.

%The conclusion holds for the vertex $4$ similarly. 

%If $3$ is a $zw$-circled, since $34$ is a $z$-stroke, i.e., $3$, $4$ are $w$-close, then the vertex $4$ must be $zw$-circled. In this case, we have $z$-circled vertice $1,2$, and $w$-circled vertice $3,4$ which are $w$-close, which contradicts with Rule 1(c).  Therefore, all vertice are $z$-circled. 

Suppose now the vertex $1$ is $zw$-circled. Estimate 2 implies that  vertex $2$ is $w$-circled  and vertex $4$ is $z$-circled. 
%Due to our assumption, the body $3$ is a circled vertex. 
If the vertex $3$ is $w$-circled, then the vertex $4$ must be also $w$-circled since the vertices $3$ and $4$ are $w$-close, and thus actually $zw$-circled. Due to Rule {(1c)}, $2$ and $3$ must be also $z$-circled since they are $z$-close. The analysis with the assumption that vertex $3$ is $z$-circled is analogous. In conclusion, all vertices are $zw$-circled. 

%\rd{For another case $12$ and $34$ are $zw$-edges, $14$, $23$ are $w$ or $z$-storkes, the conclusion holds similarly.}\marginpar{\rd{This has to be displayed}}  

\end{proof}

The following rule will be applied for $N=4$ only.  

{\bf Rule 2j: } If a 4-{vertex} diagram consists of two connected components each {consisting} of two $z$-(resp. $w$-)circled vertices connected via a $zw$-{edge}. Then these two components cannot be connected by exactly one $z$-(resp. $w$-)edge.

\begin{proof}
Suppose all {the} vertices are $z$-circled and there are $zw$-edges between $1,2$ and $3,4$ respectively. If the assertion does not hold,  without loss of generality we assume {that} {there is no other edge in the diagram but a $z$-edge between vertices $2$ and $3$.}
%that is a $z$-edge between vertices $2$ and $3$ and there is no other edge in the diagram. 
Since the vertex $1$ is only $z$-circled, there holds 
\be \label{E:Rule-2j-1} 
w_1 \simeq m_2 W_{21} - \eta m_2 Z_{21} \simeq W_{21} - \eta Z_{21} \prec \ep^{-2}.
\ee
But then
\[
w_2 \simeq m_1 W_{12} - \eta(m_1 Z_{12} + m_3 Z_{32}) = -\eta m_3 Z_{32} - m_1 (W_{21} - \eta Z_{21})  \sim -\eta m_3 Z_{32}.
\]
The last step follows from \eqref{E:Rule-2j-1}. This means $w_2 \simeq \ep^{-2}$, which is {however} not allowed since the vertex $2$ is not $w$-circled.

\end{proof}

%{\bf Rule 2h: } For any pair of vertexes $(i,j)$ which either of them connectes to others via $z$-edges, then the edge $ij$ is either  disconnected or a $zw$-edge. 
%\begin{proof}
%	Since the vertexes $i,j$ connect vertexes in $I \subset \{1,2,\ldots,N\} $. Then we have 
%	\[
%	z_i = \sigma_1 \left( \sum_{k \in I,k \neq i} m_k Z_{ki} \right), \quad w_i = \sigma_2 \left( \sum_{k \in I,k \neq i} m_k Z_{ki} \right).
%	\]
%	we have $\sigma_2 z_{ij} = \sigma_1 w_{ij}$. 
%\end{proof}

{
\subsection{{Non-equal-order} singular diagrams}

{Recall we have set $\| \mathcal{W} \| \prec \|\mathcal{Z}\| \sim \varepsilon^{-2}$.}

{We first establish in this case an estimate:}
\begin{lemma}
%Suppose a singular sequence $(\mathcal{Z},\mathcal{W})$ of \eqref{E:BC-ZW} satisfying $\|\mathcal{Z}\|=\ep^{-2}$. We have 
\be \label{E:W-norm}
\|\mathcal{W}\| \succeq \ep^2.
\ee
\end{lemma} 

\begin{proof}{This} directly follows from the estimate
\be
\|\mathcal{W}\|^3 \succeq | w_{ij}|^3 = \frac{1}{z_{ij} Z_{ij}^2} \succeq \ep^{6}.
\ee
\end{proof}

{By Equation \eqref{E:BC-ZW}, a} {non-equal-order} singular sequence $(\mathcal{Z},\mathcal{W})$ {satisfies}
\be 
\begin{split}
	\frac{z_j}{\|\mathcal{Z}\|} &= \frac{1}{1-\eta^2} \sum_{k \neq j} m_k \frac{Z_{kj}}{\|\mathcal{Z}\|} - \frac{\eta}{1-\eta^2} \frac{\|\mathcal{W}\|}{\|\mathcal{Z}\|} \sum_{k \neq j} m_k\frac{W_{kj}}{\|\mathcal{W}\|}, \\
	\frac{w_j}{\|\mathcal{W}\|} &= \frac{1}{1-\eta^2} \sum_{k \neq j} m_k \frac{W_{kj}}{\|\mathcal{W}\|} - \frac{\eta}{1-\eta^2} \frac{\|\mathcal{Z}\|}{\|\mathcal{W}\|} \sum_{k \neq j} m_k \frac{Z_{kj}}{\|\mathcal{Z}\|}.
\end{split}
\ee 
Pushing $\| \mathcal{Z}\| \sim \ep^{-2} \rightarrow \infty$ along the singular sequence, we have 
\be 
\frac{z_j}{\|\mathcal{Z}\|} \sim \frac{1}{1-\eta^2} \sum_{k \neq j} m_k \frac{Z_{kj}}{\|\mathcal{Z}\|}, \quad	 \sum_{k \neq j} m_k \frac{Z_{kj}}{\|\mathcal{Z}\|} \prec 1.
\ee  
{Thus}
\be \label{E:BC-unequal}
z_j \simeq \sum_{k \neq j} m_k Z_{kj} \prec \ep^{-2}.
\ee
{Thus only $z$-strokes in the singular diagrams are visible and there are no $z$-circles. Therefore we get the following rule:}

{\bf Rule 3a: } There is at least one $z$-{edge} in the singular diagram. At each end of a $z$-{edge} there is necessarily another $z$-{edge}. {There are no $z$-circles.}
%\begin{proof}
%%These conclusions immediately follows from $\|\mathcal{Z}\| = \ep^{-2}$ and Equation \ref{E:BC-unequal}.
%\end{proof}

\begin{proposition} \label{P:Est-dis}
 There is no maximal $z$-stroke. {That is}{, we have} $r_{jk} \prec 1$, if {the} vertices $i$ and $j$ are connected via a $z$-stroke. 
\end{proposition}
 \begin{proof}
{Suppose there is a $z$-stroke between the vertices $i$ and $j$, so $Z_{ij}$ is of leading order.} By {assumption} $\ep^{-2} \sim Z_{ij} = r_{ij}^{-3} z_{ij}$, we have $r_{ij} \prec 1$  since $z_{ij} = z_i - z_j \prec \ep^{-2}$,  {as {there are no $z$-circles} in the diagram}. 
\end{proof}

 }

\section{Singular Diagrams of balanced configuration for $N=4$} \label{S:N=4}

In this section, we identify possible singular diagrams for $4$ bodies in the plane.

\subsection{Equal-order singular diagrams} 
 As in \cite{AlbouyKaloshin2012}, by a bi-colored vertex we mean a vertex from which there {emanates} at least one $z$-stroke and one $w$-stroke. We define $C$ as the maximal number of strokes from a {bi-colored vertex}. Clearly $2 \leq C \leq 6$.

% In the rest of this section, we let $N=4$.

\subsubsection{No bi-colored vertex} 

Assume no bi-colored vertex exists. By Rule (1a), there must exist a $z$-stroke between, say, vertices $1, 2$ and a $w$-stroke, which must be between the vertices $3, 4$. By Rule (1a, 1f), the vertices $1, 2$ are double-circled. By Rule (1e), the $z$-stroke between the vertices $1, 2$ is maximal. Thus $z_{12} \sim \ep^{-2}$ and $w_{12} \sim \ep^{2}$. Hence the $w$-circled vertices $1,2$ are $w$-close. But this {contradicts} Rule (1c) since $1, 2$ {consist} an isolated component. Therefore a possible diagram {must have} at least one bi-colored vertex.

\subsubsection{$\mathbf{C=2}$}

Suppose there is a $zw$-edge between vertices $1$ and $2$. Rule (2g) asserts the existence of another $zw$-edge, which must be between $3$ and $4$, and there are no other strokes. Thus the vertices $1, 2$ form an isolated component, which contradicts Rule (2a).

Suppose between vertices $1, 2$ there is a $z$-{edge} and between vertices $2, 3$ there is a $w$-{edge}. By Rule (2c), there is no stroke between the vertices $1$ and $3$. The vertex $4$ can be only connected to $1$ or $3$ via a single-colored edge due to the assumption $C=2$.

Suppose the vertex $4$ is connected to both vertices $1$ and $4$. By Rule (2h), the edge between $3, 4$ is a $z$-{edge} and the edge between $1, 4$ is a $w$-{edge}.  
By Rule (2d), all vertices must be circled. With Rule (2i) we have 
\begin{itemize}
	\item Singular Diagram I(a,b): A quadrilateral with a pair of opposite $z$-edges and a pair of opposite $w$-edges. All vertices are $w$-circled called (I(a)) or $z$-circled (I(b)).
	\item Singular Diagram II: A quadrilateral with a pair of opposite $z$-edges and a pair of opposite $w$-edges. All vertices are $zw$-circled.
\end{itemize}
Now assume there is no {stroke} emanating from the vertex $4$. So the vertices $1,2,3$ form an isolated {triangle}. Due to Rule (1g) the vertices $1$ and $3$ are $zw$-circled. Then the vertex $2$ is $w$-close to vertex $1$ and $z$-close to vertex $3$, implying that it is also $zw$-circled by Rule (1b). Therefore, we have 
\begin{itemize}
	\item {Singular Diagram III:} A ${\rm V}$-shape diagram and an isolated vertex. The vertices $1, 2, 3$ are $zw$-circled, and are connected in a ${\rm V}$-shape isolated component with {a} $z$-{edge} between vertices $1, 2$ and {a} $w$-{edge} between vertices $2, 3$. The isolated vertex $4$ is not circled. 
\end{itemize}

If there is an edge between vertices $3$ and $4$, and there is no edge between the vertices $1, 4$, we claim that the edge between $3, 4$ must be a $z$-{edge}.

Suppose this is not true{, and the }vertices $3$ and $4$ are connected via a $w$-{edge}. By Rule (1g), the vertices $1$ and $4$ are both $zw$-circled. Since the vertices $2, 3$ and $3, 4$ are connected only via $w$-{edge}s, they are $z$-close {to each other. T}hus $2,3$ are $z$-circled. Since the vertices $1$ and $2$ are $w$-close, $2$ is also $w$-circled. Moreover, due to Rule (2e), the vertex $3$ is also $w$-circled. So all vertices are $zw$-circled.

We claim that the $w$-stroke between vertices $3$ and $4$ is not maximal, that is, $r_{34} \prec 1$. Otherwise we would have $w_{34}\simeq  \ep^{-2}, z_{34} \simeq \ep^2$ on one hand, and due to \eqref{E:BC-ZW}, 
$$z_{34}  \simeq m_2 W_{23} + m_4 W_{43} - m_3 W_{34} \simeq w_{34} $$
 on the other hand.
%\[
%z_{34}  \sim m_2 W_{23} + m_4 W_{43} - m_3 W_{34} \sim w_{34} \sim \ep^{-2}, 
%\]
%which is a contradiction. 

Therefore, by Rule (1e), the $w$-stroke between the vertices $2, 3$ is maximal: $w_{23} \simeq  \ep^{-2}$, which implies  $r_{23} \simeq 1$. By Rule (1e) again, the $z$-stroke between the vertices $1, 2$ is also maximal, so $z_{12} \simeq \ep^{-2}$, implying $r_{12} \simeq 1$.  

Since the vertices $2,3,4$ are $z$-close{, we} have $z_{13},z_{14} \simeq \ep^{-2}$. On the other hand, since the vertices $1,2$ are $w$-close, the vertices $3,4$ are $w$-close and $w_{23} \simeq  \ep^{-2}$, we have $w_{13},w_{14},w_{24} \simeq \ep^{-2}$, which in turn imply $r_{13},r_{14} \simeq \ep^{-2}$.  By Estimate 1, we have $z_{24} \succeq \ep^{2}$. Thus $r_{24} \succeq 1$. In conclusion, we have the following estimates on the mutual distances
\[
r_{34} \prec 1, r_{23},r_{12} \simeq 1, r_{13}, r_{14} \simeq \ep^{-2}, r_{24} \succeq 1,
\]
in which only $r_{34} \to 0$. This is disallowed by Rule (2g). So we get the claim that {the edge} between the vertices $3$ and $4$ {is a $z$-edge}.

%Now we suppose vertices $3$ and $4$ connecting with each other via a $z$-stroke exactly.
By Rule (1g), {the} vertices $1$ and $4$ are both $zw$-circled. The vertex $2$ is $w$-close to $1$ and the vertex $3$ is $w$-closed to $4$. Therefore $2, 3$ are  $w$-circled. Moreover, since $2$ and $3$ are $z$-close, they are both either $w$-circled or $zw$-circled. We get the following possible singular diagrams:
\begin{itemize}
\item Singular diagram IV(a,b): $\Pi$-shape diagram. Vertex $1$ is connected via a $z$-{edge} to vertex $2$, Vertex $2$ is connected via a $w$-{edge} to vertex $3$, vertex $3$ is connected via a $z$-{edge} to vertex $4$. Vertices $1, 4$ are {$zw$-}circled. Vertices $2, 3$ are $w$-circled. This is called IV(a);

Vertex $1$ is connected via a $w$-{edge} to vertex $2$, Vertex $2$ is connected via a $z$-{edge} to vertex $3$, vertex $3$ is connected via a $w$-{edge} to vertex $4$. Vertices $1, 4$ are {$zw$}-circled. Vertices $2, 3$ are $z$-circled. This is called IV(b).

	\item Singular diagram V(a,b):  $\Pi$-shape diagram. No $zw$-edges. Adjacent edges are of different types. All vertices are $zw$-circled. {The diagram with two maximal $z$-strokes is called V-1(a). The diagram with only {one} maximal $z$-stroke is called V-2(a).  
		%		   The diagram with two maximal $z$-strokes is called (V-1(a)). 
The diagram with two $w$-strokes are called V-1(b) and V-2(b) respectively. }
\end{itemize}

\subsubsection{$\mathbf{C=3}$}
%Due to the symmetry of the rules, it suffices to consider the following alternatives 
The vertex with three strokes may either be connected to  {three} other vertices via single-color edges, or connected to one of them via a single-colored edge and to another via a $zw$-edge. We consider the following two representative cases:
\begin{enumerate}
	\item The vertex $3$ is connected to vertices $2, 4$ via $z$-{edges}, and to vertex $1$ via a $w$-{edge}. 
	\item  There is a $zw$-edge connecting $1$ and $2$, while there is another $z$-{edge} between the vertices $2$ and $3$.
\end{enumerate}
\vspace{0.3 cm}

For Case (1), by Rule (2c) there are no edges between the vertices $1, 4$. If there is an edge between $2, 4$, then by Rule (2c) this must be a $z$-stroke. By Rule (1g), the vertex $1$ is  $zw$-circled. Thus the vertex $3$ is $z$-circled. The vertices $2, 3, 4$ are all $w$-close to each other, thus by Rule (1d) all of them are $w$-circled, and the $w$-stroke between $1, 3$ is maximal.  Therefore the $z$-circled vertices $1,3$ are $z$-close. By Rule (2f), the vertices $2,4$ are $z$-close. By Rule (1d), the vertices $2$ and $4$ must both be $z$-cricled. By Rule (1e), the $z$-strokes between $2, 3$ and between $3, 4$ are maximal.
Thus
$$r_{13},r_{34},r_{32} \simeq 1.$$
By Rule (2f), the vertices $2, 4$ are $z$-close to each other and thus 
$$r_{24} \prec 1.$$
By Rule (1d) all of the vertices are $z$-circled. 

We have  
$$z_{13}\prec \ep, z_{23},z_{34} \simeq \ep^{-2}, w_{13} \simeq \ep^{-2}, w_{23},w_{34}  \prec \ep,$$
thus
$$r_{12} =(z_{12}w_{12})^{1/2} \simeq \ep^{-2}, r_{14} = (z_{14}w_{14})^{1/2} \simeq \ep^{-2}.$$

Since only $r_{24 } \to 0$ among all {$r_{ij}\, 's$}, this is disallowed by Rule (2g).

If the vertices $2$ and $4$ are not connected by any stroke, then by Rule (1g), the vertices $1,2,4$ are all $zw$-circled. Since $3$ is $z$-close and $w$-close to  $zw$-circles, it is also $zw$-circled. This gives

\begin{enumerate}
\item Singular diagram VI(a): The vertex $3$ is connected to vertices $2, 4$ via $z$-{edges}, and to vertex $1$ via {a} $w$-{edge}. There {are} no other strokes {in the diagram}. All the vertices are $zw$-circled.
\item {Singular diagram VI(b): {The diagram obtained by exchanging} the colors of strokes {and circles} from Singular diagram VI(a).}
\end{enumerate}

%For case 1, since $13$ and $23$ are $z$-stroke, then $1,2,3$ are $w$-close. We claim that there is no stroke connecting $4$ with $1$ or $2$ since
%\[
%z_{24} =z_{23} + z_{34} \sim z_{23} > \ep, \quad w_{24} = w_{23} + w_{34} > \ep.
%\]
%Due to Rule 1a, the vertex $4$ is both $z$ and $w$-circled, which implies $1,2,3$ are $w$-circled since they are $w$-close. On the other hand, since $W_{34}$ is a $w$-stroke, $34$ is $z$-close. We have body $3$ is $z$-circled. Then there is another $z$-circle due to Rule 1d. We suppose $1$ is a maximal $z$-stroke hence $z_{13} \sim \ep^{-2}$.    

%If there is no another stroke emanatin from $2$, we have $2$ is also $z$-circled. 

%If there is a stroke from $2$ to $1$, since $w_{21} <\ep$ it can not be a $w$-stroke. Therefore $Z_{12}$ is a $z$-stroke. If $2$ is not $z$-circled, we have $z_{23} <\ep^{-2}, z_{12} < \ep^{-2}$, which contradicts with 
%\[
%\ep^{-2} \sim z_{13} =z_{12} + z_{23} <\ep^{-2}.
%\] 
%\vspace{0.5 cm}

For Case (2), we have
\[
z_{12} \simeq w_{12} \simeq \ep, \quad w_{23} \prec \ep. 
\]
By Rule (2c), the vertices $1$ and $3$ are not connected via any edge. 
%On the other hand, $13$ can not be a $zw$-edge due to Rule 2b. 
Rule (2g) implies the existence of another $zw$-edge, which can only be between the vertices $3, 4$. All the vertices  are $w$-close to each other.
% $2,3$ are $w$-close and $3,4$ are $w$-close, then all vertices are $w$-close to each others.
By Rule (1c), there are no $w$-circles in the diagram. By Rule (2j), there must be a stroke connecting vertices $1$ and $4$. Moreover, by Rule (2h), if there is {an edge }between the vertices $1$ and $4$, then it must be a $z$-{edge}. {We obtain:}

\begin{enumerate}
\item Singular Diagram VII(a,b): Quadrilateral {diagrams}. A pair of opposite edges are $zw$-edges. Then (a) the other pair of opposite sides are $z$-{edges}. All the vertices are $z$-circled; or (b) the other pair of opposite sides are $w$-{edges}. All the vertices are $w$-circled. 
\end{enumerate}

\subsubsection{ $\mathbf{C=4}$}

Suppose the vertex $3$ is a bicolored vertex with $4$-strokes. Then we have the following possibilities:
\begin{enumerate}
	\item There are $zw$-edges between the vertices $1, 3$ and between the vertices $2, 3$;
	\item There is a $zw$-edge between the vertices $1, 3$. Between the vertices $2, 3$ there is a $z$-{edge}. Between the vertices $3,4$ there is a $w$-{edge}. 
	\item There is a $zw$-edge between the vertices $1, 3$. Between the vertices $2, 3$ and $3, 4$ there are single-colored {edge}s of the same color. 
	%\item  There is a $zw$-edge between the vertices $1, 3$. Between the vertices $2, 3$ and $3, 4$ there are $w$-strokes. 
\end{enumerate}
%It suffices to consider the case (3) only because of the symmetry between case (3) and case (4). \\

For (1), Rule (2b) implies that there is a $zw$-edge between the vertices $1, 2$. Since there are no maximal strokes, by Rule (1e) there are no circles at the vertices $1, 2, 3$. The vertex $4$ is isolated, which is not circled by Rule (1e). %there is no circle in the isolated component hence $4$ is also not circled. 
We have
\begin{enumerate}
\item Singular diagram VIII: Three vertices form a triangle connected by $zw$-edges. The other vertex is isolated. No circles.
 \end{enumerate}

For (2), by Rule (2c), there are no edges between the vertices $1, 2, 4$. This is in contradiction with Rule (2g) and therefore no diagram in this case.
%we have 
%\[
%z_{13}  \sim w_{13} \sim \ep, \ep^2 \leq w_{23},z_{34} < \ep,
%\]
%that is $z_{14} \sim w_{12} \sim \ep$. Then neither $z$ stroke nor $w$-stroke connecting the vertex $1$ to vertex $2$, $4$. 

%Due to Rule 2g, the edge $24$ is $zw$-edge. Due to Rule 1a, the vertex $1$ are either $z$- or $w$-circled. We assume $1$ is a $z$-colored vertex, then the vertex $3$ is $z$-colored since $z_{13},w_{13} \sim \ep$. Note that $34$ is $w$-stroke, then vertexes $3$ and $4$ are $z$-close. Moreover $2$ and $3$ are $z$-close. This contradicts with Rule 1e, therefore NO singular diagram remains. \\

%If $24$ is neither $z$ nor $w$-edge,   Otherwise $13$ or $14$ would be $zw$-edge hencec $23$ or $34$ also $zw$-edge due to Rule 2b.  We have $1$ must be circled
%\[
%\begin{pmatrix}
%	z_1 \\
%	w_1
%\end{pmatrix} = \mathbf{B} \begin{pmatrix}
%%	m_3 W_{31} 
%\end{pmatrix}
%\]   
%Suppose $1$ is $z$-circled without loss of generality, then $3$ and $4$ are also $z$-circled since they are $z$-close. Moreover, $z_{23}$ is a maximal $z$-stroke and body $2$ is $z$-circled. Note that
%\[
%\begin{pmatrix}
%	z_2 \\
%	w_2
%\end{pmatrix} = \mathbf{B} \begin{pmatrix}
%	m_3 Z_{32} \\
%	0 
%\end{pmatrix}
%\]
%$2$ is also $w$-circled hence $1,3$ and $4$ are $w$-circled. $W_{43}$ is a maximal $w$-stroke. Then one has $z_{24} \sim \ep^{-2}$ and $w_{24} \sim \ep^{-2}$. There are neither $z$ or $w$ stroke connecting the vertex $2$ and $4$. \sout{That is the only possible diargam.} {\color{blue} Impossible due to Rule 2g}   \\

For (3), Rule (2g) asserts the existence of another $zw$-edge. However this cannot be between any pairs of the vertices $1, 2, 4$ by Rule (2c). So no diagram exists in this case as well.

\subsubsection{ $\mathbf{C=5}$}

Suppose there are $5$ strokes emanating from the vertex $3$. Without loss of generality we assume that between the vertices $1, 3$ and the vertices $2, 3$ there are  $zw$-edges, and there is a $z$-stroke between the vertices $3, 4$. By Rule (2b), there must exist  a $zw$-edge between the vertices $1, 2$ closing the triangle with vertices $1, 2, 3$. Rule (1a) implies that there is either another stroke from the vertex $4$, or vertex $4$ is circled. 

We claim that there is no {edge} connecting vertex $4$ to vertex $1$ or $2$. To see this, without loss of generality we assume that there is an edge between the vertices $1$ and $4$. This cannot be a $zw$-edge by Rule (2b) and the assumption $C=5$. A single-colored stroke connecting $4$ to $1$ or $2$ will result in a triangle of edges with different {types} of sides, which is disallowed by Rule (2c).

Thus vertex $4$ must be circled, and indeed {$zw$}-circled by Rule (1g). All vertices are connected by $z$-strokes, and the only $w$-strokes are within $zw$-edges, which are never maximal. Thus all vertices should be $w$-circled by Rule (1b). But this contradicts Rule (1d). 

Therefore, {no singular diagram for $C=5$}.

\subsubsection{$\mathbf{C=6}$}

{In this case, $6$ strokes emanate from, say, vertex $1$. Therefore, between the vertices $1, 2$, the vertices $1, 3$ and the vertices $1, 4$ there are $zw$-edges. By Rule (2b), there are also $zw$-edges between other vertices and there are no maximal strokes. By Rule (1e), there are no circles. We get 

\begin{enumerate}
\item {Singular diagram IX: The fully-edged diagram. Any pairs of vertices are connected by $zw$-edges. There are no-circles.}
 \end{enumerate}
%the fully-edged diagram as Singular diagram IX.}

\begin{figure}[htbp]
	\centering
	\includegraphics[scale=0.5]{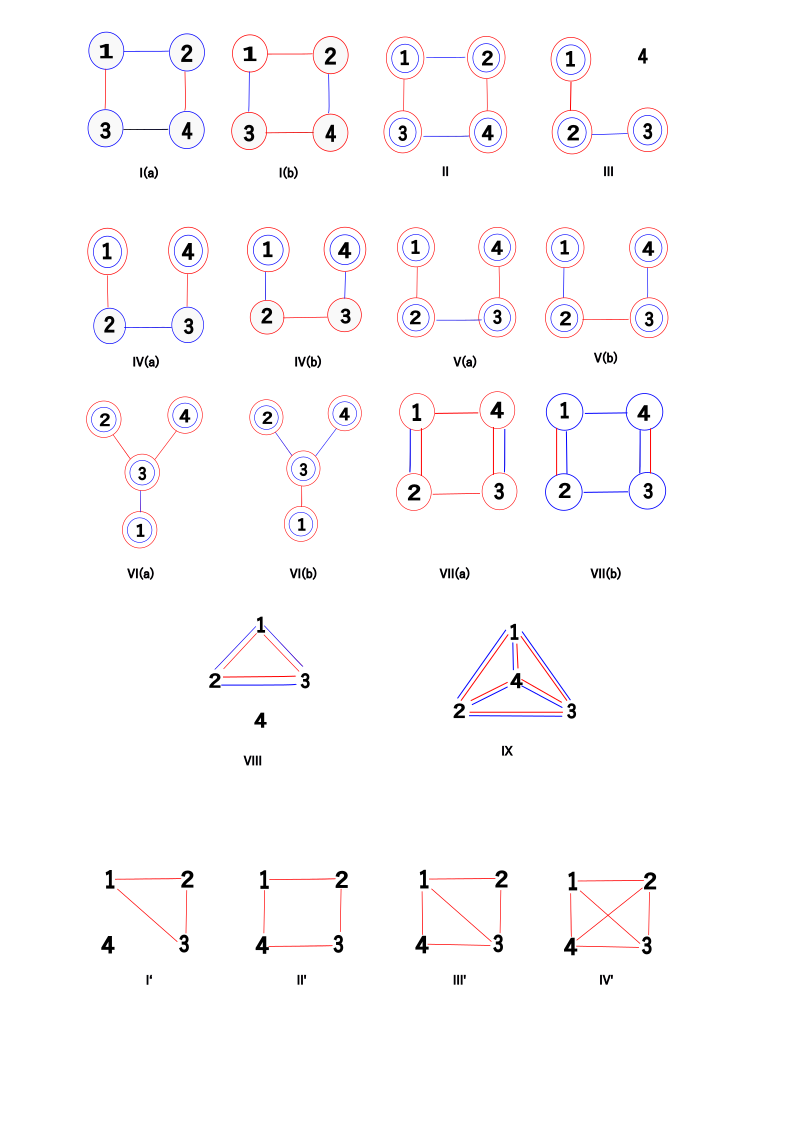}
	\caption{Singular Diagrams}
	\label{figure}
\end{figure}

\subsection{{Non-equal-order} singular diagrams}

Now we analyze {non-equal-order} singular diagrams. {Again we denote} {by} $C$ the maximal number of strokes emanating from a vertex. Due to Rule (3a), we {have} $2 \leq C \leq 3$. 

\subsubsection{$C=2$} If the vertices $1,2$ are connected via a $z$-{edge}, {then by Rule (3a) there is a $z$-{edge} between, say, vertices $2$ and $3$. There must be some $z$-{edges} emanating from $3$. When the vertex $3$ is connected to $1$ via a $z$-{edge}, there can be no other $z$-{edges} in the diagram.  We get {the} singular diagram $I^\prime$. If vertex $3$ is connected to vertex $4$, then by Rule (3a) there must be a $z$-{edge} between the vertices $4$ and $1$. We get {the} singular diagram $II^\prime$.}

\subsection{$C=3$} {We continue $II^\prime$. If the only addition is a $z$-{edge} between, say, the vertices $1, 3$, we get {the} singular diagram $III^\prime$. Otherwise we get the fully-edged diagram, which is the singular diagram $IV^\prime$.}

\section{{Remaining singular diagrams and their corresponding mass conditions}} \label{S:mass}

{Several singular diagrams remain after a systematic exclusion made in Section \ref{S:N=4} when $\mathcal{Z}$ and $\mathcal{W}$ are of equal order.} In this section, we identify explicit mass conditions for each of these diagrams, except for the last singular diagram IX.

{We begin with stating an identity that must be satisfied for $\mathbf{S}$-balanced configurations where $\eta \neq 0$.} Denote by 
\[
f_k = \sum_{j \neq k} m_j r_{jk}^{-3} (q_j - q_k).
\]
By \eqref{E:BC-1}, we have 
\[
\bigg(\begin{pmatrix}
	1+ \eta & 0 \\
	0 & 1- \eta 
\end{pmatrix} q_j \bigg) \wedge q_j = f_j \wedge q_j = \sum_{k \neq j} m_k r_{kj}^{-3} q_k \wedge q_j.
\]
A bi-vector of the form $q_i\wedge q_j$ can be represented by an anti-symmetric $2\times 2$ matrix after choosing an orthonormal basis in the plane{,} which can be subsequently identified to a scalar by choosing the upper-right element. With this in mind we compute 
\[
\bigg(\begin{pmatrix}
	1+ \eta & 0 \\
	0 & 1- \eta 
\end{pmatrix} q_j \bigg) \wedge q_j = 2 \eta x_j y_j = \frac{\eta}{2\sqrt{-1}} \left(z_j^2 -w_j^2 \right).
\]
By summarizing all these items from $1$ to $N$, we {have}
\[
\begin{split}
\sum_{j=1}^n m_j \bigg(\begin{pmatrix}
	1+ \eta & 0 \\
	0 & 1- \eta 
\end{pmatrix} q_j \bigg) \wedge q_j  & =  \sum_{j=1}^n m_j f_j \wedge q_j \\
& = \sum_{j=1}^n \sum_{k \neq j} m_k m_j r_{kj}^{-3} q_k \wedge q_j=0,
\end{split}
\]
{so} we get the {identity}
\be \label{E:zw-weight}
\sum_{j=1}^n m_j (z_j^2 - w_j^2) =0.
\ee

\subsection{$C=2$}

{The remaining singular diagrams are I(a,b), II, III, IV(a,b), V-1(a,b), V-2(a,b).}

\subsubsection{Singular diagram I} We {only} consider singular diagram I(a) {here.} The mass condition is the same for singular diagram I(b) by an analogous argument.

Since all the vertices are $w$-circled only, we have
\be \label{E:I-a}  
\begin{split}
      m_2 Z_{21} &  - \eta m_4 W_{41}  \sim (1-\eta^2)z_1 \prec \ep^{-2}, \\
	  m_1 Z_{12} &  -  \eta m_3 W_{32}  \sim (1-\eta^2)z_2 \prec \ep^{-2}, \\
	   m_4 Z_{43} &  - \eta m_2 W_{23}  \sim (1-\eta^2)z_3 \prec \ep^{-2}, \\
	  m_3 Z_{34} &  -  \eta m_1 W_{14}  \sim (1-\eta^2) z_4 \prec \ep^{-2}. \\
	 \end{split} 
\ee
Eliminating $Z_{12}$, we get
\be \label{E:Singular-1-a}
 m_1 m_4 W_{41} \sim  -m_2 m_3 W_{32}.
\ee
The vertices $1$ and $2$ are $w$-close, as well as vertices $3$ and $4$. By Rule (1c), we have  
\[
(m_1 + m_2) w_1 +  (m_3+m_4) w_3  \prec \ep^{-2}. 
\]
In the diagram, the vertex $1$ is $w$-circled with a $w$-edge to vertex $4$ and a $z$-edge to vertex $2$. By Equation \eqref{E:BC-ZW}, using \eqref{E:I-a}, we have 
\[
 w_1 \sim  \frac{1}{1-\eta^2} \left(m_4 W_{41} - \eta m_2 Z_{21}\right)  \sim m_4 W_{41}, 
\]
as well as 
\[
 w_3 \sim \frac{1}{1-\eta^2} \left(m_2 W_{23} - \eta m_4 Z_{43} \right) \sim m_2 W_{23}.
\]
{Thus} we get
\be \label{E:Singular-1-b}
(m_1+m_2) m_4 W_{41} \sim  (m_3+m_4) m_2 W_{32}. 
\ee
{Equations} \eqref{E:Singular-1-a} and \eqref{E:Singular-1-b} imply
\be \label{E:mass-I-1}
m_3(m_1+m_2) + (m_3+m_4)m_1 = 0.
\ee
Moreover, due to \eqref{E:zw-weight}, {we have}
\[
(m_1+m_2) w_1^2 \sim - (m_3+m_4) w_3^2.
\]
Thus
\[
(m_1+m_2)m_4^2 W_{41}^2 \sim - (m_3+m_4) m_2^2 W_{23}^2.
\]
Squaring \eqref{E:Singular-1-a}, we have  
\[
m_1^2 m_4^2 W_{41}^2 \sim m_2^2 m_3^2 W_{23}^2.
\]
Thus
\[
(m_1+m_2)m_3^2 + m_1^2(m_3+m_4) = 0.
\]
In conclusion, we get the mass condition 
\be \label{E:mass-I} 
\begin{cases}
 m_3(m_1+m_2) + (m_3+m_4)m_1 = 0, \\
 (m_1+m_2)m_3^2 + m_1^2(m_3+m_4) = 0.
\end{cases}
\ee

\subsubsection{Singular diagram II} The vertices $1,3$ are $w$-close, as well as vertices $2,4$. The vertices $1,2$ are $z$-close, as well as vertices $3,4$. Hence there exists $\ep^{-2} \simeq z_0,w_0 \in \mathbb{C}$ such that 
\be \label{E:II-A-1}
\begin{split}
w_1,w_3 & \sim -(m_2+m_4)w_0, \quad w_2,w_4 \sim (m_1+m_3) w_0, \\
z_1,z_2 & \sim (m_3 +m_4) z_0, \quad z_3,z_4 \sim -(m_1 + m_2) z_0.
\end{split}
\ee
Since the $zw$-circled vertex $1$ is connected to vertex $3$ via a $z$-edge and to vertex $2$ via a $w$-edge, by Equation \eqref{E:BC-ZW} we get
\be
\begin{split}
z_1 & \sim  \frac{1}{1-\eta^2} m_3 Z_{31} - \frac{\eta}{1-\eta^2} m_2 W_{21}, \\
w_1  & \sim  \frac{1}{1-\eta^2} m_2 W_{21} - \frac{\eta}{1-\eta^2} m_3 Z_{31}.
\end{split}
\ee
Eliminating $W_{21}$ we get 
\be  \label{E:II-A}
z_1 + \eta w_1 \sim m_3 Z_{31}.
\ee
Similarly, we have 
\be  \label{E:II-B}
z_3 + \eta w_3 \sim m_1 Z_{13}.
\ee
Combining \eqref{E:II-A} and \eqref{E:II-B}, we have
\[
m_1(z_1 + \eta w_1) + m_3 (z_3 + \eta w_3) \prec \ep^{-2}.
\]
This together with \eqref{E:II-A-1} imply
\be \label{E:sigma-1}
(m_1m_4 -m_2m_3) z_0 \sim \eta (m_1+m_3)(m_2+m_4) w_0.
\ee
Applying an analogous argument to the vertices $3$ and $4$, we get
%\[
%m_4 W_{43} \sim w_3 + \eta z_3, \;\; m_3 W_{34} \sim \eta z_4 + w_4. 
%\] 
%Then 
%\[
%m_3(w_3 + \eta z_3) + m_4(\eta z_4 +w_4) \sim 0
%\]
%Together with \eqref{E:II-A-1}, there is 
\be \label{E:sigma-2}
(m_1 m_4-m_2m_3) w_0 \sim \eta (m_1+m_2)(m_3+m_4) z_0.
\ee
Combining \eqref{E:sigma-1} and \eqref{E:sigma-2}, we get the mass condition 
\be \label{E:mass-II}
(m_1m_4-m_2m_3)^2 = \eta^2 (m_1+m_2)(m_3+m_4)(m_1+m_3)(m_2+m_4).
\ee
We note that {Identity} \eqref{E:zw-weight} leads to the same mass condition \eqref{E:mass-II}.

\subsubsection{Singular diagram III}
In this singular diagram, vertices $1,2$ are $w$-close, vertices $2$ and $3$ are $z$-close. Thus there exists $ \ep^{-2} \simeq z_0,w_0 \in \mathbb{C}$ such that 
\[
\begin{split}
	z_2,z_3 & \sim m_1 z_0, \quad z_1 \sim -(m_2 +m_3)z_0{,}\\
	w_3 & \sim -(m_1 + m_2) w_0, \quad w_1,w_2 \sim m_3 w_0{.}
\end{split}
\]
From the diagram and Equation \eqref{E:BC-ZW}, we get   
\[
w_1  \sim -\frac{\eta}{1-\eta^2} m_2 Z_{21}, \quad  z_1 \sim \frac{1}{1-\eta^2} m_2 Z_{21}.
\]
Thus $w_1 \sim -\eta z_1$ and 
\be \label{E:III-1} 
m_3 w_0 \sim \eta (m_2 + m_3) z_0.
\ee 
Again from the diagram and Equation \eqref{E:BC-ZW}, we have
\[
w_3  \sim \frac{1}{1-\eta^2} m_2 W_{23}, \quad  z_3 \sim -\frac{\eta}{1-\eta^2} m_2 W_{23}.
\]
Thus $z_3 \sim  -\eta w_3$ and 
\be \label{E:III-2}
m_1 z_0 \sim  \eta(m_1+m_2) w_0.
\ee
Combining equations \eqref{E:III-1} and \eqref{E:III-2}, we have 
\be \label{E:mass-III}
\eta^2 (m_1 + m_2)(m_2 + m_3) = m_1m_3.
\ee
{Identity} \eqref{E:zw-weight} leads to the same mass condition.

\subsubsection{Singular diagram IV}

In this diagram, the only $z$-circled vertices are $1$ and $4$. Thus {by Rule (1c)} {we have}
\[
m_1 z_1 + m_4 z_4 \prec \ep^{-2},
\]
and there exists $\ep^{-2} {\simeq} z_0 \in \mathbb{C}$, {in which $z_{0}$ represents a sequence}, such that 
\be \label{E:IV-A} 
z_1 \sim -m_4 z_0, \quad z_4 \sim m_1 z_0.
\ee
Moreover, since the vertices $1$ and $2$ {are connected {via a $z$-edge}, using Equation \eqref{E:BC-ZW}} we have  
\be \label{E:IV-zw} 
z_1 \sim \frac{1}{1-\eta^2} m_2 Z_{21}, \quad w_1\sim -\frac{\eta}{1-\eta^2} m_2 Z_{21}.  
\ee
{Thus} $w_1 \sim -\eta z_1$. Similarly, $w_4 \sim -\eta z_4$. 

Since the vertices $1$ and $2$ are $w$-close, as well as vertices $3$ and $4$,  {by Rule (1c) we have}
\[
(m_1+m_2) w_1 + (m_3 + m_4) w_4 \prec \ep^{-2}{.}
\]
Thus we obtain
\be \label{E:IV-B}
(m_1+m_2) z_1 + (m_3 + m_4) z_4 \prec \ep^{-2}{.}
\ee
{Combining} \eqref{E:IV-A} and \eqref{E:IV-B}, we have $(m_1 + m_2) m_4 = (m_4+m_3) m_1$ {and thus}
\be 
 {m_1 m_3=m_2 m_4.}
\ee
{On the other hand, using the identity \eqref{E:zw-weight} we get} 
\be 
m_1z_1^2 + m_4 z_4^2 \sim (m_1+m_2) w_1^2 + (m_3+m_4) w_4^2 \sim \eta^2 \left( (m_1+m_2) z_1^2 + (m_3+m_4) z_4^2 \right). 
\ee
Then by \eqref{E:IV-A} {again} we obtain
\be
m_1 m_4^2 + m_4 m_1^2 = \eta^2 \left((m_1+m_2)m_4^2 + (m_3+m_4)m_1^2\right). 
\ee
In summary, the singular diagram IV leads to the co-dimension 2 mass condition 
\be \label{E:mass-IV}
\begin{cases}
 m_2 m_4  =m_3 m_1, \\
 m_1 m_4^2 + m_4 m_1^2 = \eta^2 \left((m_1+m_2)m_4^2 + (m_3+m_4)m_1^2\right).
\end{cases}
\ee

\subsubsection{Singular diagram V}

%For the singular diagram V, there are indeed two different types of singularties, due to whether both of $z$-strokes are maximal $z$-strokes.    

The vertices $1$ and $2$ are $w$-close, as well as vertices $3$ and $4$. As the center of mass is {put} at the origin, there exists $w_0   \in \mathbb{C}, w_{0} \simeq  \ep^{-2}$ such that  
\[
w_1,w_2 \sim (m_3 + m_4) w_0,\quad w_3,w_4 \sim-(m_1+m_2) w_0.
\]  
Applying the same argument leading to \eqref{E:IV-zw} to the pairs of vertices $1, 2$ and $3, 4$ respectively, we obtain the relations
\be \label{R: zz}
z_1 \sim -\frac{1}{\eta} w_1 \sim  -\frac{1}{\eta} (m_3 + m_4) w_0, \quad z_4 = -\frac{1}{\eta} w_4 \sim  \frac{1}{\eta} (m_1 + m_2) w_0.
\ee

{By Rule (1e), there exists maximal $z$-stroke and maximal $w$-stroke in the diagram. Thus the $w$-edge between the vertices $2, 3$ is maximal. Therefore }
\[
w_{14} =w_{12}+w_{23}+w_{34} \simeq \ep^{-2}. 
\]
Using Equation \eqref{E:BC-ZW}, we have% implies 
%\be
%-\eta \left(m_2 Z_{21} - m_3 Z_{34} \right) \sim w_{14} \simeq \ep^{-2}.
%\ee
%Now we have 
\[
z_{14} = z_1 -z_4  = \frac{1}{1-\eta^2} \left(m_2 Z_{21} - m_3 Z_{34} \right) \sim -\frac{1}{\eta} w_{14} \simeq \ep^{-2},
\]
so the vertices $1$ and $4$ are not $z$-close.

We consider the following cases: (1) only the $z$-{edge} between $1, 2$ is maximal; (2) only the $w$-{edge} between $3, 4$ is maximal; (3) both $z$-{edges} are maximal.

{In case (1), the vertices $3, 4$ are $z$-close. The vertices $2, 3$ are also $z$-close since the $w$-stroke between them is maximal. By Rule (1c) we have}
\[
m_1 z_1 + (m_2 + m_3 + m_4) z_4  \prec \ep^{-2}.
\]
{This together with \eqref{R: zz} implies}
\be \label{E:mass-V-1-a}
m_1(m_3+m_4)=(m_2 + m_3 +m_4)(m_1 + m_2).
\ee
From \eqref{E:zw-weight}, we get 
\be 
\begin{split}
& \eta^2(m_3+m_4)(m_1+m_2)(m_1+m_2+m_3+m_4) \\
&= m_1(m_3+m_4)^2 + (m_2+m_3+m_4)(m_1+m_2)^2.
\end{split}
\ee
Let $M:=m_1+m_2+m_3+m_4$ be the total mass. Combining them, we obtain the mass condition {for case (1){:}}{\footnotesize
\be \label{E:mass-V-1}
\begin{cases}
m_1(m_3+m_4) =(m_2 + m_3 +m_4)(m_1 + m_2), \\
\eta^2(m_3+m_4) (m_1+m_2) M = m_1(m_3+m_4)^2 + (m_2+m_3+m_4)(m_1+m_2)^2.
\end{cases}
\ee
}

{For case (2) we permute the vertices $1, 4$ and $2, 3$ from (1). We obtain
{\footnotesize	
\be \label{E:mass-V-3}
\begin{cases}
m_4(m_1+m_2)=( m_1 + m_2 + m_3 )(m_3 + m_4), \\
\eta^2(m_1+m_2)(m_3+m_4)M = m_4(m_1+m_2)^2 + (m_1 + m_2 + m_3)(m_2+m_4)^2.
\end{cases}
\ee
}

Consider now the case (3). {As the vertices $2, 3$ are $z$-close,  we may write $z_2,z_3 \sim z_0$, for  $z_{0} \simeq \ep^{-2}$.} From the $w$-diagram, there exists $w_0 \simeq \ep^{-2}$ such that 
\[
w_1,w_2 \sim -(m_3 + m_4) w_0, \quad w_3,w_4 \sim (m_1 + m_2) w_0.
\]
{As the vertices $1$ and $4$ are $zw$-circled, by \eqref{E:BC-ZW}} we have
\[
z_1 \sim  \frac{1}{1-\eta^2} m_2 Z_{21} \sim -\frac{1}{\eta} w_1 \sim  \frac{1}{\eta} (m_3 + m_4) w_0
\]
and 
\[
z_4 \sim \frac{1}{1-\eta^2} m_3 Z_{34} \sim -\frac{1}{\eta} w_4 \sim  -\frac{1}{\eta} (m_1 + m_2) w_0.
\]
{By Rule (1c), we have}
\[
m_1 z_1 + m_4 z_4 + (m_2 + m_3) z_0  \prec \ep^{-2},
\]
which implies 
\be %\label{E:V-zw}
z_0 \sim -\frac{1}{\eta(m_2 + m_3)} \left( m_1 m_3 - m_2 m_4 \right) w_0.
\ee
{Combining these relations with \eqref{E:zw-weight},} we get the corresponding mass condition
\be \label{E:mass-V-2}
\begin{split}
&  m_1(m_3+m_4)^2 + m_4(m_1+m_2)^2 + \frac{(m_1m_3 -m_2 m_4)^2}{m_2 + m_3}  \\
& = \eta^2 \left((m_1+m_2)(m_3+m_4)^2 + (m_3+ m_4)(m_1+m_2)^2 \right). 
\end{split}
\ee

\subsection{$C=3$} 

For $C=3$, we have two types of diagrams VI(a,b), VII(a,b).

\subsubsection{Singular diagram VI(a, b)} {We consider VI(a). The diagram VI(b) leads to the same mass condition.}

In VI(a), the vertices $2,3,4$ are $w$-close and {the $w$-stroke between the vertices $1, 3$ is maximal. Thus} there exists $w_0$ such that
\be \label{E:VI-a}
w_2,w_3,w_4 \sim -m_1 w_0, \quad w_1 \sim (m_2 + m_3+m_4)w_0.
\ee
Moreover, we have 
\be \label{E:VI-b}
\begin{split}
	z_2 & \sim \frac{1}{1-\eta^2} m_3 Z_{32},\;\;w_2\sim -\frac{\eta}{1-\eta^2}m_3 Z_{32}{,} \\
	z_4 & \sim \frac{1}{1-\eta^2} m_3 Z_{34}, \;\; w_4 \sim -\frac{\eta}{1-\eta^2} m_3 Z_{34}{,}  \\
	z_1 & \sim -\frac{\eta}{1-\eta^2} m_3 W_{31}, \;\; w_1\sim \frac{1}{1-\eta^2} m_3 W_{31}{.} 
\end{split}
\ee
Since vertex $3$ and vertex $1$ are $z$-close, we have 
\[
m_2 z_2 +m_4 z_4\sim  -(m_3 + m_1)z_1.
\]
Combining {this} with \eqref{E:VI-a} and \eqref{E:VI-b}, we get
\be \label{E:mass-VI-a}
 m_1(m_4 +  m_2) = \eta^2 (m_1+m_3)(m_2 + m_3+m_4).
\ee
On the other hand, {it follows from \eqref{E:zw-weight} that}
\[
\eta^2(m_2 + m_3 +m_4)(m_1^2 + m_1(m_2+m_3+m_4) -\eta^2(m_1+m_3)(m_2+m_3+m_4)) = (m_2+m_4) m_1^2.
\]
{Inserting} \eqref{E:mass-VI-a} into the equation above, we  {get exactly} the same mass condition. Thus {singular diagram VI(a,b) gives the mass condition }
\be \label{E:mass-VI}
m_1(m_2+m_4) = \eta^2(m_1 + m_3)(m_2 +m_3 +m_4).
\ee
}

\subsubsection{Singular diagram VII(a, b)} 
{Again we only consider the case VII(a). }
In this diagram, the vertices $1, 2$ and $3, 4$ are connected via $zw$-edges. This implies that the vertices $1$ and $2$ are both $z$ and $w$-close, as well as the vertices $3$ and $4$. 

{We assume the $z$-edge between the vertices $1, 4$ is not maximal: $z_{14} \prec \ep^{-2}$, then
\[
z_{23} =z_{21}+z_{14} +z_{43} \prec \ep^{-2},
\]
%namely vertices $2$ and $3$  should be also $z$-close and the $z$-edge between $2, 3$ would not be maximal as well.
so there are no maximal $z$-{edges} in the diagram. This contradicts Rule (1e). Thus both $z$-edges between $1, 4$ and between $2, 3$ are maximal.}

{As the vertices $1$ and $2$ are {only} $z$-circled,  we get from Equation \eqref{E:BC-ZW} that}
\[
\begin{split}
z_1 & \sim \frac{1}{1-\eta^2} (m_2 Z_{21} + m_4 Z_{41} ) -\frac{\eta}{1-\eta^2} m_2 W_{21} \simeq \ep^{-2},  \\
w_1 & \sim   \frac{1}{1-\eta^2} m_2 W_{21} -\frac{\eta}{1-\eta^2} (m_2 Z_{21} + m_4 Z_{41} ) \prec \ep^{-2}, \\
z_2 & \sim  \frac{1}{1-\eta^2} (m_1 Z_{12} + m_3 Z_{32} ) -\frac{\eta}{1-\eta^2} m_1 W_{12} \simeq \ep^{-2}, \\
w_2 & \sim \frac{1}{1-\eta^2} m_1 W_{12} -\frac{\eta}{1-\eta^2} (m_1 Z_{12} + m_3 Z_{32} ) \prec \ep^{-2}.
\end{split}
\]
{By eliminating the terms containing $Z_{kl}$ we get }
\[
z_1 \sim \frac{1}{\eta} m_2 W_{21}, \quad 
z_2 \sim   \frac{1}{\eta} m_1 W_{12}.
\]
As the vertices $1$ and $2$ are $z$-close, we have
\[
z_{12} =z_1 -z_2 \sim \frac{1}{\eta}(m_1+m_2) W_{21} \prec \ep^{-2}. 
\]
This is only possible if $m_1 + m_2=0$. Applying the same argument to the vertices $3$ and $4$, we get the following mass condition:
\be \label{E:mass-VII} 
m_1 + m_2=0,\quad m_3+m_4 =0.
\ee
{The mass condition deduced from Equation \ref{E:zw-weight} is covered by this condition.} 

\subsection{$C=4$} 

{We {consider}  the singular diagram VIII.}

{As the {vertices} $1,2,3$ are connected {via $zw$-edges} and no circle in this diagram, from Equation \eqref{E:BC-ZW} we have the following equation at vertex $1$:}
\be 
\begin{split}
(m_2 Z_{21} + m_3 Z_{31}) - \eta  (m_2 W_{21} + m_3 W_{31}) &  \prec \ep^{-2}, 
\\
(m_2 W_{21} + m_3 W_{31}) - \eta  (m_2 Z_{21} + m_3 Z_{31}) &  \prec \ep^{-2}.
\end{split}
\ee
These equations are equivalent to 
\be
m_2 Z_{21} + m_3 Z_{31} \prec \ep^{-2}, \quad m_2 W_{21} + m_3 W_{31}  \prec \ep^{-2}.
\ee
Similarly, by {applying the same consideration} at vertices $2$ and $3$ {we obtain}, 
\be
\begin{split}
& m_3 Z_{32} + m_1 Z_{12} \prec \ep^{-2}, \quad m_3 W_{32} + m_1 W_{12}  \prec \ep^{-2}, \\
& m_1 Z_{13} + m_2 Z_{23}  \prec \ep^{-2}, \quad m_1 W_{13} + m_2 W_{23}  \prec \ep^{-2}.
\end{split}
\ee
{As $Z_{12}, Z_{23}, Z_{31} \simeq \ep^{-2}$, these imply} 
\be
\frac{Z_{21}}{m_3} \sim \frac{Z_{31}}{m_2} \sim \frac{Z_{23}}{m_1}, \quad \frac{W_{21}}{m_3} \sim \frac{W_{31}}{m_2} \sim \frac{W_{23}}{m_1}.
\ee
Since $Z_{ij} =z_{ij}^{-\frac{1}{2}} w_{ij}^{-\frac{3}{2}}, W_{ij} =w_{ij}^{-\frac{1}{2}} z_{ij}^{-\frac{3}{2}}$, there exists $\rho_0 \in \mathbb{C}^*$ such that $z_{ij}= \rho_0 w_{ij}$ holds for any pair {of $(i, j), 1 \leq i\neq j \leq 3$}. {By Substitution} we {get} 
\[
m_3 z_{21}^2 \sim \pm m_2 z_{31}^2 \sim  \pm m_1 z_{23}^2.
\]
Suppose $m_k = \mu_k^2$, we have 
\[
\mu_1 z_{23} \sim \ep_2\mu_2 z_{31} \sim \ep_3 \mu_3 z_{21}, \quad \ep_2^4=\ep_3^4=1.
\]
{Substituting these into $z_{12}+z_{23}+z_{31}=0$ {and taking into consideration the positivity of the masses}, we obtain} three alternatives {forming the corresponding} mass condition.
\be \label{E:mass-VIII}
\frac{1}{\sqrt{m_j}} = \frac{1}{\sqrt{m_k}} + \frac{1}{\sqrt{m_l}}, \quad  \text{ $\{j,k,l\}$ is a permutation of $\{1,2,3\}$}. 
\ee

\subsection{$C=6$} 

{We {can not} get any explicit mass condition in this case. {We shall  exclude this singular diagram} by a different argument.}

\section{Finiteness Results} \label{S:Finitness}

We have analyzed all possible equal-order singular diagrams, \emph{i.e.} {those satisfying} $\|\mathcal{Z}\| \simeq \|\mathcal{W}\|$. All but the last one of them give explicit mass conditions of co-dimension at least one. Let us first exclude this last case. To this end, recall the definition of $\mathbf{S}$-moment of inertia 
\[
I_\mathbf{S}(\mathbf{q})= \sum_{j=1}^{{N}}  m_j \mathbf{S}q_j \cdot q_j.
\]
We have  
\begin{proposition} \label{P:Non-dominited}
The $\mathbf{S}$-moment of inertia $I_\mathbf{S}$ can not take infinitely many values along a singular sequence. 
\end{proposition}	

\begin{proof}

$\mathbf{S}$-{balanced} configurations are critical points of the potential function $U$ restricted to the inertia ellipsoid $\{I_\mathbf{S}(\mathbf{q})=1\}$, see \cite{Alain2011, Moeckel2014}. 

We have 
\be 
\partial_{x_i} I_\mathbf{S} =-2 \partial_{x_i} U, \;\; \partial_{y_i} I_\mathbf{S} = -2 \partial_{y_i} U. 
\ee  
This implies $df=0$ on the variety $\mathcal{A}$ defined by \eqref{E:BC-1} where $f = I_\mathbf{S} + 2U$. {By} \cite[Prop 3.6]{Mumford1976} (see also \cite{AlbouyKaloshin2012}),  $f$ is not dominating and can therefore only take finitely many values.

Moreover, by multiplying $q_j$ on both sides of \eqref{E:BC-1} and summarizing {over all $j$}, we {obtain} $I_\mathbf{S}$ on the left hand side, and {the} potential function $U$ on the right hand side by {Euler's} homogeneous function theorem. Thus $I_\mathbf{S}=U$ and $f = 3I_\mathbf{S}$ on $\mathcal{A}$. {Thus $I_\mathbf{S}$ can only take finitely many values. }
%The conclusion follows. 
\end{proof}
%In particular, if we take a singular sequence from a connected continum containing real balanced configurations, we have $I_\mathbf{S}=I_0>0$ and $I_\mathbf{S}$ can not tend to $0$.    

\begin{proposition}  \label{S:Prop-1}
There are no singular sequences realizing singular diagram {IX}.
\end{proposition}

\begin{proof}

In the singular diagram IX, all the edges are $zw$-{edges}. Thus we {have} $z_{ij} \simeq w_{ij} \simeq \ep$. {Consequently we have}
\[
\begin{split}
I_\mathbf{S}(\mathbf{q}) & = \sum_{i=1}^N m_i((1+\eta)x_i^2 +(1-\eta)y_i^2 ) \\
& = \frac{1}{\sum_{i=1}^N m_i} \sum_{1 \leq i<k \leq N} {\bigg(} m_im_k(1+\eta)x_{ik}^2 + m_im_k(1-\eta)y_{ik}^2 {\bigg)}  \\
& = \frac{1}{\sum_{i=1}^N m_i}  \sum_{1 \leq i<k \leq N} m_km_i \bigg( \frac{(1+\eta)}{4}(z_{ik}^2 + 2z_{ik}w_{ik} + w_{ik}^2) \\
& \quad \quad - \frac{(1-\eta)}{4}(z_{ik}^2 - 2z_{ik}w_{ik} + w_{ik}^2) \bigg) \\
& =  \frac{1}{\sum_{i=1}^N m_i}  \sum_{1 \leq i<k \leq N} m_km_i \left(  \frac{\eta}{2} (z_{ik}^2 +w_{ik}^2) + z_{ik}w_{ik} \right) \simeq \ep^2.
\end{split}
\]
Thus $I_\mathbf{S}(\mathbf{q}) \rightarrow 0$ and takes {infinitely} many values. This contradicts Proposition \ref{P:Non-dominited}.
\end{proof}

As compared to the case of central configurations, the {presence} of non-equal-order singular diagrams complicates the analysis for balance configurations potentially {when $\eta$ is not so small}. {However, if we restrict our attention to equal-order singular sequences}, the singular diagram IX has been excluded by Proposition \ref{S:Prop-1}, {and} the singular diagrams I-VIII all give explicit mass conditions  of co-dimension $1$ or $2$. We {first {state} a conditional {finiteness proposition}, where the exceptional mass set $\mathcal{M}_\eta$ is given {by the mass conditions \eqref{E:mass-II}, \eqref{E:mass-III},\eqref{E:mass-IV}, \eqref{E:mass-V-2}, \eqref{E:mass-VI}, \eqref{E:mass-VIII}}}. 
 
\begin{proposition} \label{S:Generic}
For masses $(m_1,m_2,m_3,m_4) \in \mathbb{R}_+^4\setminus \mathcal{M}_\eta$, if there are infinitely many balanced configurations, then their singular sequences are {necessarily not of} {equal order}.
\end{proposition}
\begin{proof}
Suppose there are no non-equal-order singular {sequences}. Assume there are {infinitely} many $\mathbf{S}$-balanced configurations for the four-body problem, an {equal-order} singular sequence exists and lives in at least one of the singular {diagrams} I-IX. Because the singular diagram IX is excluded by Proposition \ref{S:Prop-1} and the mass {conditions} \eqref{E:mass-I}, \eqref{E:mass-V-1}, \eqref{E:mass-VII} are {irrelevant for positive masses, a} singular sequence occurs if one or more of the mass conditions 
\eqref{E:mass-II}, \eqref{E:mass-III},\eqref{E:mass-IV}, \eqref{E:mass-V-2}, \eqref{E:mass-VI}, \eqref{E:mass-VIII} holds. This defines the set of exceptional masses $\mathcal{M}_\eta$.

%Then we get the exception mass set $\mathcal{M}_\eta$ which consists of at least co-dimension $1$ varieties.  
\end{proof}
In particular, for any given masses $(m_1,m_2,m_3,m_4) \in \mathbb{R}_+^4$, {we only have to consider \eqref{E:mass-II} and \eqref{E:mass-VIII} for sufficiently small $\eta>0$, as we may always choose $\eta$ sufficiently small to avoid the other mass conditions.} \\

The following lemma states that {non-equal-order} singular sequences are nevertheless irrelevant for sufficiently small $\eta>0$.

\begin{lemma} \label{L:Exclude-unequal} 
For any given mass $(m_1,m_2,m_3,m_4) \in \mathbb{R}^4_+$, there exists $\eta_0>0$ such that for any $0<\eta<\eta_0${,} all possible singular sequences of {$\mathbf{S}$-}balanced configurations must be of {equal order}.	
\end{lemma}
Assume there is a singular sequence {$\{\mathcal{Z},\mathcal{W}\}=\{(\mathcal{Z}^{(n)},\mathcal{W}^{(n)})\}_{n}$} {of \eqref{E:BC-ZW}}. Each element {$\{(\mathcal{Z}^{(n)},\mathcal{W}^{(n)})\}$} satisfies 
\be \label{E:BC-normal}
\begin{split}
	\frac{z_j}{\|\mathcal{Z}^{(n)}\|} &= \frac{1}{1-\eta^2} \sum_{k \neq j} m_k \frac{Z_{kj}}{\|\mathcal{Z}^{(n)}\|} - \frac{\eta}{1-\eta^2} \frac{\|\mathcal{W}^{(n)}\|}{\|\mathcal{Z}^{(n)}\|} \sum_{k \neq j} m_k\frac{W_{kj}}{\|\mathcal{W}^{(n)}\|}, \\
	\frac{w_j}{\|\mathcal{W}^{(n)}\|} &= \frac{1}{1-\eta^2} \sum_{k \neq j} m_k \frac{W_{kj}}{\|\mathcal{W}^{(n)}\|} - \frac{\eta}{1-\eta^2} \frac{\|\mathcal{Z}^{(n)}\|}{\|\mathcal{W}^{(n)}\|} \sum_{k \neq j} m_k \frac{Z_{kj}}{\|\mathcal{Z}^{(n)}\|}{,}
\end{split}
\ee
and {the} $\eta$-independent equations 
\be \label{E:ZW-ij} 
(z_i-z_j)(w_i-w_j)^3 Z_{ij}^2=1, \quad (w_i-w_j)(z_i-z_j)^3 W_{ij}^2=1.
\ee
{Assume that there are only finitely many central configurations, i.e., when $\eta=0$, this system of equations has only finitely many solutions counted up to dilation $(\mathcal{Z},\mathcal{W}) \mapsto (\alpha \mathcal{Z}, \alpha^{-1}\mathcal{W})$  for any $\alpha \in \mathbb{C}^*$. }We have the following {Lemma}. 
\begin{lemma} \label{L:Continuous}
	We take a sequence of $\eta=\eta_m \rightarrow 0$ when $m \rightarrow \infty$. {For each $m \in \mathbb{N}$, }suppose $\{\mathcal{Z}_m^{(n)},\mathcal{W}_m^{(n)}\}$ is a singular sequence of Equation \eqref{E:BC-ZW} satisfying $\| \mathcal{W}\| \preceq \| \mathcal{Z}\| { = \ep_n^{-2} \rightarrow \infty}$ as { $n \rightarrow \infty$}. For any fixed $n$, %sufficiently large,
	there exists a subsequence of {$\{\mathcal{Z}_m^{(n)},\mathcal{W}_m^{(n)}\}$}, still denoted by $\{\mathcal{Z}^{(n)}_m, \mathcal{W}^{(n)}_m\}_{m \geq 1}$, such that %they solves Equation \eqref{E:BC-ZW} with $\eta = \eta_m \rightarrow 0$ and 
	%the following conditions holds
	\[
	\frac{\mathcal{Z}^{(n)}_m}{\|\mathcal{Z}^{(n)}_m\|} \rightarrow 
	\frac{Z^*}{\|Z^*\|}, \quad \frac{\mathcal{W}^{(n)}_m}{\|\mathcal{W}^{(n)}_m\|} \rightarrow \frac{W^*}{\|W^*\|} \quad \text{ as $m \rightarrow \infty$},
	\]
	{hold, in} which $(Z^*,W^*)$ corresponds to a central configuration of the $N$-body problem.
\end{lemma}

\begin{proof}
	%Consider the singular sequence $\{(\mathcal{Z}^{(n)},\mathcal{W}^{(n)})\}$ 
	%\[
	%\mathcal{Z}^{(n)} =(z_1,\ldots,z_n,Z_{12},\ldots,Z_{n(n-1)}), \; \mathcal{W}^{(n)} =(w_1,\ldots,w_n,W_{12},\ldots,W_{n(n-1)}).
	%\]
	{We re-write Equation \eqref{E:BC-normal}} into a perturbative {form:}
	\be 
	\begin{split}
		\frac{z_j}{\|\mathcal{Z}^{(n)} \|} & = \sum_{k \neq j} m_k \frac{Z_{kj}}{\|\mathcal{Z}^{(n)}\|} + \frac{\eta}{1-\eta^2} \sum_{k \neq j} m_k \left( \eta \frac{Z_{kj}}{\|\mathcal{Z}^{(n)}\|} - \frac{W_{kj}}{\|\mathcal{W}^{(n)}\|} \frac{\|\mathcal{W}^{(n)}\|}{\|\mathcal{Z}^{(n)}\|} \right), \\
		\frac{w_j}{\|\mathcal{W}^{(n)} \|} & = \sum_{k \neq j} m_k \frac{W_{kj}}{\|\mathcal{W}^{(n)}\|} + \frac{\eta}{1-\eta^2} \sum_{k \neq j} m_k \left(\eta \frac{W_{kj}}{\|\mathcal{W}^{(n)}\|} - \frac{Z_{kj}}{\|\mathcal{Z}^{(n)}\|} \frac{\|\mathcal{Z}^{(n)}\|}{\|\mathcal{W}^{(n)}\|} \right){.}
	\end{split}
	\ee
	%where 
	%\[
	%\begin{split}
	%	\|\mathcal{Z}^{(n)}\|:&= \max\{|z_1|,\ldots,|z_n|,|Z_{ij}|\}, \\
	%	\|\mathcal{W}^{(n)}\|:&= \max\{|w_1|,\ldots,|w_n|,|W_{ij}|\}.
	%\end{split}
	%\]  
	%are the $L^\infty$ norm of the sequences.

	{The vectors} $\frac{\mathcal{Z}^{(n)}}{\|\mathcal{Z}^{(n)} \|}, \frac{\mathcal{W}^{(n)}}{\|\mathcal{W}^{(n)} \|}$ are {all} on the unit sphere of $\left(\mathbb{C}^{\frac{\mathcal{N}}{2}},\|\cdot \| \right)$. %For fixed $n$, 
	{There} exists {$\ep_n \to 0$} such that  $\|\mathcal{Z}^{(n)}\| = \ep_n^{-2}$. {By} \eqref{E:W-norm}, {we have}  $\|\mathcal{W}^{(n)} \| \succeq \ep_n^2$. {Thus}
	$\frac{\|\mathcal{Z}^{(n)}\|}{\|\mathcal{W}^{(n)}\|} \preceq \ep_n^{-4}$ { the right hand of which is independent of small $\eta$} as {$n \to \infty$}.

	Now let $\eta = \eta_m \rightarrow 0$ as $m \rightarrow \infty$. By compactness of the product of unit spheres, we get a subsequence $\left(\frac{\mathcal{Z}^{(n)}_m}{\|\mathcal{Z}^{(n)}_m\|}, \frac{\mathcal{W}^{(n)}_m}{\|\mathcal{W}^{(n)}_m\|}\right)$ converging to a point $(\frac{ \mathcal{Z}^*}{\|\mathcal{Z}^*\|}, \frac{ \mathcal{W}^*}{\|\mathcal{W}^*\|})$ for which {$(\mathcal{Z}^*,\mathcal{W}^*)$} solves the system
	\be \label{E:CC-normal}
	\begin{split}
		\wt z_j^* & = \sum_{k \neq j} m_k \wt Z_{kj}^*, \\
		\wt w_j^* & = \sum_{k \neq j} m_k \wt W_{kj}^*{,}
	\end{split}
	\ee
	%\rd{$(\mathcal{Z}^*,\mathcal{W}^*)$}
	{as well as $(\lambda \mathcal{Z}^*,\lambda^{-1} \mathcal{W}^*)$. This enables us to freely choose the norms $\|\mathcal{Z}^*\|, \|\mathcal{W}^*\|$.}%satisfying the constraint \eqref{E:ZW-ij}.

	%$\mathcal{\wt Z}^* =(\wt z_1,\ldots,\wt z_n, \cdots, \wt Z_{ij},\cdots)$ and $\mathcal{\wt W}^*$ is defined respectively which are on the unit sphere of  $\left(\mathbb{C}^{\frac{\mathcal{N}}{2}}, \|\cdot \| \right)$, 
	%\rd{Thus} there exist $\mathcal{Z}^*,\mathcal{W}^* \in \mathbb{C}^{\frac{\mathcal{N}}{2}}$ such that 
	%\[
	%\mathcal{ \wt Z}^* = \frac{ \mathcal{Z}^*}{\|\mathcal{Z}^*\|}, \; \mathcal{ \wt W}^* = \frac{ \mathcal{W}^*}{\|\mathcal{W}^*\|}.
	%\]
	In particular, %considering equations \eqref{E:ZW-ij}, we choose  $\mathcal{Z}^*,\mathcal{W}^*$ such that 
	%when $m \rightarrow \infty$. 
	because $\|\mathcal{W}^{(n)}\| \preceq \|\mathcal{Z}^{(n)} \|=\ep_n^{-2}$ and $\|\mathcal{W}^{(n)}\| \succeq \ep_n^2$, we have $ 1 \preceq \|\mathcal{Z}_m^{(n)} \|  \| \mathcal{W}_m^{(n)} \|  \preceq \ep_n^{-4}$, where $\ep_n$ is given which is independent of $m$. Therefore up to subsequence we can assume that {the limit} $\|\mathcal{Z}_m^{(n)} \|  \| \mathcal{W}_m^{(n)} \|$ as $m \rightarrow \infty$ {exists}. We choose  $\|\mathcal{Z}^* \|$ and  $\| \mathcal{W}^*\|$ such that 
	\[
	\frac{ \|\mathcal{Z}_m^{(n)} \| \cdot \| \mathcal{W}_m^{(n)} \|} { \|\mathcal{Z}^* \| \cdot  \| \mathcal{W}^*\|} \rightarrow 1 \text{ as  $m \rightarrow \infty$}.
	\]

	Since $\mathcal{Z}_{m}^{(n)}$ and $\mathcal{W}_{m}^{(n)}$ satisfy  \eqref{E:ZW-ij}, we get 
	\[
	\frac{1}{\|\mathcal{Z}_m^{(n)}\|^3 \cdot \|\mathcal{W}_m^{(n)}\|^3}=\frac{(z_i-z_j)(w_i-w_j)^3Z_{ij}^2}{\|\mathcal{Z}_m^{(n)}\|^3 \cdot\|\mathcal{W}_m^{(n)}\|^3}   \rightarrow  
	\frac{(z_i^*-z_j^*)(w_i^*-w_j^*)^3 (Z_{ij}^*)^2}{\|\mathcal{Z}^*\|^3 \cdot \|\mathcal{W}^*\|^3}
	. 
	\]
	as $m\rightarrow \infty$. This implies
	\[
	(z_i^*-z_j^*)(w_i^*-w_j^*)^3 (Z_{ij}^*)^2 =1.
	\]
	Analogously, we have 
	\[
	(w_i^*-w_j^*)(z_i^*-z_j^*)^3 (W_{ij}^*)^2 =1.
	\]
	
	Then $(\mathcal{Z}^*,\mathcal{W}^*)$ solves the equations \eqref{E:CC-normal} and satisfies \eqref{E:ZW-ij}, therefore
	{it} corresponds to {a} central configuration.

\end{proof}
There are only finitely many central configurations up to similitudes in four-body problem in the plane with positive mass \cite{AlbouyKaloshin2012,HamMoecke2006} and in five-body problem with all positive mass{es} {except perhaps for those contained} in a co-dimension $2$ algebraic variety \cite{AlbouyKaloshin2012}. A consequence of Lemma \ref{L:Continuous} is that, {if we assuming the finiteness of central configurations in the plane for a given set of positive masses, then there} exists a finite set of complex numbers 
$$\{ {C_{kj}, D_{kj} \in \mathbb{C}}: 0<|C_{kj}|,|D_{kj}| \leq 1 \}$$ {consisting of} all limiting points of
$\left(\frac{\mathcal{Z}^{(n)}_m}{\|\mathcal{Z}^{(n)}_m\|}, \frac{\mathcal{W}^{(n)}_m}{\|\mathcal{W}^{(n)}_m\|}\right)$. For any $C_{kj},D_{kj}$ from this set, we assume by taking subsequence and {by abuse of notation} that 
\be \label{E:conv} 
\frac{Z_{kj}^{(n)}}{\|\mathcal{Z}^{(n)}\|} \rightarrow C_{kj}, \;\;  \frac{W_{kj}^{(n)}}{\|\mathcal{W}^{(n)}\|} \rightarrow D_{kj}.
\ee
{Then}
\be \label{E:Distance}
(r_{kj}^{(n)})^{-4} = Z_{kj}^{(n)} W_{kj}^{(n)} \rightarrow C_{kj} D_{kj} \|\mathcal{Z}^{(n)}\| \cdot \|\mathcal{W}^{(n)}\| \text{ as $m \rightarrow \infty$}.
\ee
Now we consider the limit procedure {as} $n \rightarrow \infty$. Recall that for any singular sequence there always hold $\|\mathcal{Z}\| {\cdot} \|\mathcal{W}\|  \succeq 1$. We are in the position to prove {Lemma} \ref{L:Exclude-unequal}.

\begin{proof}[Proof of Lemma \ref{L:Exclude-unequal}]
{Assume the conclusion does not hold for a sequence $\{\eta_m\}_{m=1}^\infty$ satisfying $\eta_m \rightarrow 0$ as $m \rightarrow \infty$. , i.e., for each $\eta_m \rightarrow 0$ there exists} a {non-equal-order} singular sequence  $\{\mathcal{Z}^{(n)},\mathcal{W}^{(n)}\}$ satisfying $\| \mathcal{W}\| \prec \| \mathcal{Z}\| \sim \ep^{-2}$ of system \eqref{E:BC-ZW}. Lemma \ref{L:Continuous} implies the pointwise convergence of the sequences to central configurations as ${\eta_{m}} \rightarrow 0$.  

%Our strategy is based on reduction to absurdity. 

%Assume such a singular sequence exists, then we can always pick up a sub-sequence of this singular sequence which is not allowed. 

{As}  $\|\mathcal{Z}\| {\cdot} \|\mathcal{W}\|  {\succeq} 1${, up to taking subsequence, we consider  the following two possibilities:}

If  $\|\mathcal{Z}\| {\cdot} \|\mathcal{W}\|  \succ 1$, we have by \eqref{E:Distance} {that any mutual distance} $r_{kj} \rightarrow 0$. Consider the {$\mathbf{S}$-moment} of inertia 
\[
\begin{split}
	I_\mathbf{S}(\mathbf{q})& = \sum_{i=1}^N m_i \left((1+\eta) x_i^2 + (1-\eta) y_i^2 \right) \\
	& = \frac{1}{\sum_{i=1}^N m_i} \sum_{1 \leq i <j \leq N} m_i m_j r_{ij}^2 + \eta \left( \sum_{i=1}^N m_i (x_i^2 -y_i^2) \right) \\
	&  = \frac{1}{\sum_{i=1}^N m_i} \sum_{1 \leq i <j \leq N} m_i m_i r_{ij}^2 + \frac{\eta}{2} \sum_{i=1}^N m_i(z_i^2 + w_i^2). 
\end{split}
\]
{By Identity \eqref{E:zw-weight}, we have}
\be \label{E:BC-moment}
\begin{split}
	I_\mathbf{S}(\mathbf{q})&  = \frac{1}{\sum_{i=1}^N m_i} \sum_{1 \leq i <j \leq N} m_i m_i r_{ij}^2  + \eta \sum_{i=1}^N m_iw_i^2 \\
	& =  \frac{1}{\sum_{i=1}^N m_i} \sum_{1 \leq i <j \leq N} m_i m_i r_{ij}^2 + \eta {\cdot} \frac{1}{\sum_{i=1}^N m_i} \sum_{1 \leq i <j \leq N} m_i m_i w_{ij}^2.   
\end{split}
\ee
Clearly the first part on the right hand side of \eqref{E:BC-moment} converges to $0$. We shall show the second term {also} converges to $0$. 

If {this} does not hold, we have at least one of $w_{ij} \succeq 1$. This also implies $\|\mathcal{W}\| \succeq 1$ by the definition $w_{ij}=w_i-w_j$. Thus we have $r_{ij} \preceq \ep^{\frac{1}{2}}$ due to \eqref{E:Distance}. {Then} 
\be \label{E:W-norm-1} 
\|\mathcal{W}\| \succeq W_{ij} =r_{ij}^{-3}w_{ij} \succeq \ep^{-\frac{3}{2}}.
\ee
Using \eqref{E:Distance} again we get $r_{ij} \preceq \ep^{\frac{7}{8}}$. Hence $\|\mathcal{W}\| \succeq \ep^{-\frac{21}{8}} \succ \ep^{-2}$ by \eqref{E:W-norm-1}, which {contradicts} $\|\mathcal{W} \| \prec \|\mathcal{Z}\| \sim \ep^{-2}$. {Therefore} $I_\mathbf{S}$ converges to $0$ along such {a} singular sequence. But this {contradicts} Proposition \ref{P:Non-dominited}. 

If $\|\mathcal{Z}\| \cdot \|\mathcal{W}\|  \simeq 1$, {we have that} all $r_{kj} \simeq 1$. {This is impossible, since by Proposition \ref{P:Est-dis}, in all the remaining non-equal-order singular diagrams $I', II', III', IV'$, the $r_{kj}$'s correspond to $z$-edges tend to zero. }

%This is also impossible for {non-equal-order} singular sequences, since at least three of them tend to $0$ as $\ep \rightarrow 0$. \rd{Indeed, this follows from the analysis on remaining singular diagrams.  }

\end{proof}

\section{Proof of {the main result}} \label{S:Proof}

{Analogous to  \cite[Prop 3]{AlbouyKaloshin2012}, we have}
\begin{lemma}\label{L: Product zero}
In the singular diagram VIII, 
%In a \rd{circle-free} diagram of $4$ vertices, if there are $zw$-edges between any pair of a set of $N-1$ vertices \rd{which form an isolated component,} then 
the product of any two non-adjacent distances goes to zero. 
\end{lemma}
%This brings a straightforward generalization of Proposition 3 of \cite{AlbouyKaloshin2012}. 
The proof goes the same {as in \cite{AlbouyKaloshin2012} and} we include it for completeness.

\begin{proof}
{The vertices $1, 2, 3$ form an isolated component. The vertex $4$ alone forms an isolated component. The center of mass condition implies}
\be \label{E: 6.1}
{\sum_{k=1}^{3}m_k z_k \simeq z_4, \quad\sum_{k=1}^{3}m_k w_k \simeq w_4. }
\ee
The $zw$-edges between {pairs of} the vertices $1, 2, 3$ imply the estimates
 $$z_{12} \simeq w_{13} \simeq z_{23} \simeq w_{12} \simeq w_{13} \simeq w_{23} \simeq \ep.$$
%  thus $z_{kj} \simeq w_{kj} \simeq \ep, 1 \leq k,j \leq 3$,
Thus if both $z_4,w_4 \preceq \ep$, then $z_{j4} \simeq w_{j4} \simeq \ep$ by Equation \eqref{E:BC-ZW} and the diagram is IX instead of VIII. Contradiction. 

%Then we are in singular diagram (X) hence the conclusion holds straightforwardly. 
{Without loss of generality {and up to subsequence}, we assume $z_4 \preceq w_4$. Then we have $w_4 \succ \ep$ and it follows from Equation \eqref{E:BC-ZW}} that $w_{14} \sim w_{24} \sim w_{34}$. {We are in one of the following two cases:}
\begin{enumerate}

	\item 
	%$z_4 \succ \ep$ or $z_{k4} \simeq \ep, k=1,2,3$. In these cases, we always have 
	{$z_{14} \simeq z_{24} \simeq z_{34} \succeq \ep$. This includes in particular the case $z_4 \succ \ep$. In this case we have }
	%or $z_{k4} \simeq \ep, k=1,2,3$} In this case
	\[
	r_{14} \simeq r_{24} \simeq r_{34} .
	\]
We claim {$r_{14},r_{24},r_{34}$ are all bounded. Indeed we have $z_{k4} \prec w_{k4}, 1 \leq k \leq 3$ by our assumption. In Equation \eqref{E:BC-ZW} we have} 
\be \label{E:w-N}
w_4 = \frac{1}{1-\eta^2} \sum_{k=1}^3 m_{k} r_{k4}^{-3} w_{k4} -  \frac{\eta}{1-\eta^2} \sum_{k=1}^3 m_{k} r_{k4}^{-3} z_{k4},
\ee
{If {$r_{14} \simeq r_{24} \simeq r_{34}$ are unbounded}, then we would have
$$w_4 \prec w_{l4} = w_l - w_4,$$ 
for {some $l \in \{1,2,3\}$} and thus $w_4 \prec w_l$ which contradicts the assumption. }

{The conclusion now follows from these order estimations on $r_{ij}$'s.}

\item There exists $ 1 \leq i \leq 3$ such that $z_{i4} \prec \ep$. Without loss of generality we {assume} $z_{14} \prec \ep$. {Then $z_{24} \simeq z_{34} \simeq \ep$ and thus $r_{14} \prec r_{24}  \simeq r_{34}$. } 

By Equation \eqref{E:w-N}{,}
%$$\quad\sum_{k=1}^{3}m_k w_k \simeq w_4. $$

%In this case
 %the $\mathcal{W}$-diagram consists of a component and an isolated vertex $w_4$. Since the center of masses locates at the origin, we get $(M-m_4)w_1 \sim -m_4w_4$ where $M=\sum_{k=1}^4 m_4$ is the total mass. 
 
%Due to $z_{14} \prec \ep$ thus $z_{k4} \simeq \ep$ for $2 \leq k \leq N-1$. Recall $w_{14} \sim w_{24} \sim  w_{34}$. We get $r_{14} \prec r_{24}  \simeq r_{34}$ thus 
\[
w_4 \simeq r_{14}^{-3} w_{14}.
\]
%$$\quad\sum_{k=1}^{3}m_k w_k \simeq w_4. $$

Again, if $r_{14}$ is unbounded, we would have $w_4 \prec w_l$ as in the previous case, which contradicts the assumption. Thus $r_{14}$ is bounded.

%Therefore we have 
%\[
%r_{14}^3 \sim M^{-1}m_1^{-1}(M-m_1),
%\]
%i.e., the distance $r_{14}$ is bounded. Meanwhile, for other distances we have 
For other distances $r_{k4}, k=2, 3$ we have
\[
r^2_{k4} = z_{k4}w_{k4} = \left(\frac{z_{k4}}{z_{14}} \right) z_{14} w_{k4} \sim \left(\frac{z_{k4}}{z_{14}}\right) z_{14} w_{14} \preceq \frac{z_{k4}}{z_{14}}  \simeq \frac{\ep}{z_{14}}.
\]
{By Estimate \ref{Est: 1.1} we have
% and there is no stroke connecting with vertices $1$ and $4$, we have
$\ep^2 \prec z_{14}$, thus 
$$r^2_{k4} \prec \ep^{-1}, k=2, 3.$$}

The conclusion now follows by combining these estimates with 
$$r^2_{12} \simeq r^{2}_{23} \simeq r^{2}_{31} \simeq \ep^2.$$
%, for any permutation of $\{1,2,3\}$ denoted by $\{i,j,k\}$, we obtain 
%\[
%r_{ij} r_{k4} \prec \sqrt{\ep} \rightarrow 0. 
%\]
%The proof is completed.
 \end{enumerate}
\end{proof}

%he proof is based on comparsion of orders of distances. Since no circled vertices beging included,  the equations corresponding to the diagram reads
%\be 
%\sum_{k} m_k Z_{kj} - \eta    \sum_{k} m_k W_{kj} \prec \ep^{-2}, \quad  \sum_{k} m_k W_{kj} - \eta  \sum_{k} m_k Z_{kj} \prec \ep^{-2}.
%\ee
%These equations are equivalent to central configurations that 
%\be 
%\sum_{k} m_k Z_{kj} \prec \ep^{-2}, \quad  \sum_{k} m_k W_{kj}  \prec \ep^{-2}.
%\ee

Now we are {ready} to prove the main result of this paper, {based on our analysis of singular diagrams. An alternative, direct proof of this perturbative result is given below.}

\begin{proof}[Proof of Theorem \ref{T:main}]

Assume \eqref{E:BC-ZW} has infinitely many solutions. Then at least one mutual distance, say $r_{12}$, takes infinitely many values. 
%Without loss of generality, we assume that $r_{12}$. 
This means $r_{12}^2$ is a dominating polynomial and can take all values in $\mathbb{C}$ outside a finite set.

We take a singular sequence along which $r_{12} \rightarrow 0$.  This is only possible in diagram VIII and the mass needs to satisfy the condition \eqref{E:mass-VIII}. Then $r_{23}, r_{31} \rightarrow 0$ along this singular sequence. By Lemma \ref{L: Product zero} 
%the products of non-adjacent mutual distances go to $0$. In particular, 
$r_{12}^2 r_{34}^2$, $r_{13}^2 r_{24}^2$ and $r_{23}^2 r_{14}^2$ {tend} to zero{,} therefore they are dominating polynomials
%. Their squares are dominating 
and take all values in $\mathbb{C}$ outside a finite set. Therefore,  $r_{12}^2 r_{34}^2$, $r_{13}^2 r_{24}^2$ and $r_{23}^2 r_{14}^2$ {can} tend to $\infty$ along the variety $\mathcal{A}$ defined by $\mathbf{S}$-balanced configurations.

 Take a singular sequence such that $r_{12}r_{34} \rightarrow \infty$. 
By the Estimates {1 and 2}, we are in the singular diagrams I-VII {up to numeration of vertices.}  %and numerical tag of the vertices in each singular diagrams should be re-arranged suitably. 

Since for given {masses}, only {the} singular diagram II is {relevant for sufficiently small $\eta>0$.} The corresponding mass condition is 
\[
(m_1m_2 -m_3m_4)^2 = \eta^2(m_1+m_4)(m_3+m_2)(m_1+m_3)(m_2+m_4).
\]
Analogously, let $r_{13}r_{24} \rightarrow \infty${. This }is possible only if
\[
(m_1m_3 -m_2m_4)^2 = \eta^2(m_1+m_4)(m_3+m_2)(m_4+m_3)(m_2+m_1).
\] 

Let $r_{23}r_{14} \rightarrow \infty$, we have the condition
\[
(m_1m_4-m_2m_3)^2= \eta^2 (m_1+m_2)(m_3+m_4)(m_1+m_3)(m_2+m_4).
\]
All these mass conditions have to be satisfied simultaneously. We thus get
\[
m_i = m_j + O(\eta^2), 1 \leq i \neq j \leq 4.
\]
Nevertheless, this are incompatible with \eqref{E:mass-VIII} for small $0<\eta\ll 1$.  The proof is completed.

\end{proof}

\section{{Perturbative finiteness for $N$-bodies in the plane}} \label{S:N-body}

The perturbative result Theorem \ref{T:main} for {4 point masses admits an alternative, simpler proof, which can subsequently be extended to $N$ point masses in the plane, provided that the corresponding finiteness for central configurations is established.}

\begin{theorem} \label{T:main-2}
	For any positive {masses} $(m_1,\ldots,{m_N})$, suppose there are only finitely many central configurations for the {corresponding $N$-body system} in the plane {up to similitudes, then} there exists {$\eta_0>0$} such {that} for each $0<\eta<\eta_0$, the system \eqref{E:BC-1} has only finitely many solutions up to dilations in the plane.    
\end{theorem} 
 
\begin{proof}[Proof of Theorem \ref{T:main-2}]
	
	Assume there are only finitely many central configurations but {infinitely} many $\mathbf{S}$-balanced configurations in the plane for a sequence $\eta_m \rightarrow 0$ as $m \rightarrow {\infty}$. Then we {get} a singular sequence  {$(\mathcal{Z}_{m},\mathcal{W}_{m})$} for each $m \in \mathbb{N}$. 
	\begin{enumerate}
		\item If $\|\mathcal{Z}\| \simeq \|\mathcal{W}\| \sim \ep^{-2}$, {Equation} \eqref{E:Distance} implies $r_{kj} \simeq \ep$ for any $(k,j)$, $1 \leq k <  j \leq {N}$. Therefore $Z_{kj} \simeq W_{kj} \simeq \ep^{-2}${,}  {hence $z_{kj},w_{kj} \simeq \ep$}. {By Equation} \eqref{E:BC-moment}, we get $I_\mathbf{S}(\mathbf{q}) \rightarrow 0$. A contradiction.
		\item If $\|\mathcal{W}\| \prec \|\mathcal{Z}\| \sim \ep^{-2}$, the same argument in the proof of Lemma \ref{L:Exclude-unequal} exclude the case $\|\mathcal{Z}\| \cdot \|\mathcal{W}\| \succ 1$. 
		
		Therefore it is {sufficient} to consider the case $\|\mathcal{Z}\| \cdot \|\mathcal{W}\| \simeq 1$ for which all $r_{kj} \simeq 1$. {Rule (3a)} implies {that} there is at least one $z$-{edge} connecting, say, vertices $1$ and $2$. Proposition \ref{P:Est-dis} implies that $r_{12} \prec 1$. Then we get a contradiction with \eqref{E:Distance} for {all} mutual distances  $r_{kj} \simeq 1$. The proof is completed. 
	\end{enumerate}
\end{proof}
Together with the result of Albouy and Kaloshin on the generic finiteness of the five-body problem \cite{AlbouyKaloshin2012}, we easily obtain
\begin{corollary}
For any {positive} mass $(m_1,m_2,m_3,m_4,m_5) {\in \R^{5}_{+}}$ outside a co-dimension two subvariety defined in \cite{AlbouyKaloshin2012}, there exists {$\eta_0>0$} such that for any $0<\eta<\eta_0$, the system \eqref{E:BC-1} has only finitely many solutions up to dilations in the plane. 
\end{corollary}

{\bf Acknowledgement:} We thank Alain Albouy for the reference \cite{Moeckel1997}. Y. W. is supported by DFG ZH 605/1-2 and NSFC 12471110. L. Z. is supported by DFG ZH 605/1-2 and DFG ZH 605/4-1.

%\bibliographystyle{plain}
%\bibliography{reference}
	
\end{document}